\documentclass[12pt]{amsart}

\usepackage[utf8]{inputenc}
\usepackage{amsmath, amssymb, amsthm}
\usepackage{mathrsfs}          % For script fonts
\usepackage{mathtools}
\usepackage{tikz-cd}   % For commutative diagrams
\usepackage{hyperref}          % For hyperlinks
\hypersetup{
    colorlinks=true,
    linkcolor=blue,
    citecolor=blue,
    urlcolor=blue
}
\usepackage{enumitem}          % For flexible lists
\usepackage{xcolor}            % For remarks
\usepackage{geometry}          % Page layout
\usepackage[normalem]{ulem}    % For cross out line
\usepackage[numbers]{natbib}

\newtheorem{theorem}{Theorem}[section]
\newtheorem{lemma}[theorem]{Lemma}
\newtheorem{proposition}[theorem]{Proposition}
\newtheorem{corollary}[theorem]{Corollary}

\theoremstyle{definition}
\newtheorem{definition}[theorem]{Definition}

\newtheorem{remark}[theorem]{Remark}

\numberwithin{equation}{section}

\title{Enveloping algebras via motivic Hall algebras}

\author{Xinyi Feng*, Fan Xu}
\thanks{* Corresponding author.}

\address{Department of Mathematical Sciences, Tsinghua University, Beijing 100084, P. R. China}
\email{feng-xy21@mails.tsinghua.edu.cn (X. Feng)}

\address{Department of Mathematical Sciences, Tsinghua University, Beijing 100084, P. R. China}
\email{fanxu@mail.tsinghua.edu.cn (F. Xu)}



\begin{document}

\begin{abstract}
   We give a geometric realization of the whole universal enveloping algebras of Borcherds-Bozec algebras using quivers with loops via the motivic semi-derived Hall algebra approach.       
\end{abstract}

\maketitle

\setcounter{tocdepth}{1}
\tableofcontents

\section{Introduction}

\subsection{Background}\;

Given a finite quiver $ Q=(\mathbb{I},\mathbb{H},s,t) $, where $ \mathbb{I} $ is the set of vertices, $ \mathbb{H} $ is the set of arrows and $ s,t:\mathbb{H} \to \mathbb{I} $ are the source and target maps, respectively. 

\subsubsection{Realization of enveloping algebras via acyclic quivers}\;

If $ Q $ is acyclic, then there is an associated generalized Cartan matrix which yields a Kac-Moody Lie algebra with a triangular decomposition $ \mathcal{\mathfrak{g}}_Q = \mathfrak{n}_Q^+ \oplus \mathfrak{h}_Q \oplus \mathfrak{n}_Q^- $.  
A lot of work has been done to find interpretations for $ \mathfrak{g}_Q $, in particular, for the universal enveloping algebra $ U(\mathfrak{g}_Q)  $ of $ \mathfrak{g}_Q $, in terms of representation theory of $ Q $. 

As for realization of the positive part $ U(\mathfrak{n}_Q^+) $, Ringel introduced the so-called Ringel-Hall algebra approach in \cite{R3,R5}. 
He proved the existence of Hall polynomials for any representation-directed algebra $ \Lambda $ over a finite field $ \mathbb{F}_q $ and constructed an associative algebra  $ \mathcal{H}_q(\Lambda) $ over $ \mathbb{Z}[q] $. 
After evaluating at $ q=1 $, the free $ \mathbb{Z} $-submodule $ \mathcal{L}(\Lambda) $ generated by the isomorphism classes of indecomposable $ \Lambda $-modules becomes a Lie subalgebra whose universal enveloping algebra is isomorphic to $ \mathcal{H}_1(\Lambda) \otimes_{\mathbb{Z}} \mathbb{C} $ after a complex extension.      
Also, he showed that if $ \Lambda $ is representation-finite, $ \mathcal{L}(\Lambda) $ is exactly the Chevalley $\mathbb{Z}$-form of $ \mathfrak{n}^+_Q $ and then $ \mathcal{H}_1(\Lambda) $ realizes the Kostant $ \mathbb{Z} $-form of $ U(\mathfrak{n}^+_Q) $, see \cite{R4}. 

Later, using Euler characteristics of filtration varieties, Riedtmann \cite{Riedtmann94} obtained a similar result for any representation-finite algebra $ \Lambda $ over $ \mathbb{C} $, realizing the universal enveloping algebra of the Lie algebra generated by indecomposable modules. 
In general, for any finite-dimensional algebra $  \Lambda $ over $ \mathbb{C} $, Ding, Xiao and Xu generalized Riedtmann's result in \cite{DXX10} using characteristic functions of constructible sets of stratified Krull-Schmidt, see \cite[Section 3]{DXX10} for details. 

On the other hand, Lusztig introduced a geometric approach in \cite{Lusztig90canonical,Lusztig91} using perverse sheaves on the varieties of quiver representations over $ \overline{\mathbb{F}_q} $. 
Later, he constructed the semicanonical basis for $ U(\mathfrak{n}_Q^+) $ in \cite{Lusztig00semicanonical} in terms of constructible functions on the set of irreducible components of the Lagrangian varieties, which differs from the specialization of the canonical basis at $ q=1 $. 
In addition, Geiss, Leclerc and Schröer studied the properties of the dual semicanonical basis in \cite{GLS05,GLS07}, where the relevant Lagrangian varieties are interpreted as varieties of modules over preprojective algebras. 

In terms of realizing the whole enveloping algebra $ U(\mathfrak{g}_Q) $, motivated by Lusztig's work, Nakajima defined a new family of varieties called quiver varieties in \cite{Nakajima94}. 
In the subsequent work \cite{Nakajima98}, he considered the Cartesian product of quiver varieties with certain Lagrangian subvarieties inside, defined the convolution product on homology groups and then realized $ U(\mathfrak{g}_Q) $ using the corresponding geometric construction of integrable highest weight representations given in \cite{Nakajima94}. 
In particular, if $ Q $ is of finite type, Lusztig showed in \cite{Lusztig03} that any arbitrary large finite-dimensional quotient of $ U(\mathfrak{g}_Q) $ can be realized in terms of constructible functions on the triple variety. 

Also, there have been many attempts to recover the whole quantum group, or the whole enveloping algebra, via Hall algebra approaches. 
Instead of working with abelian categories, one may turn to derived categories, or more generally, triangulated categories. 
Toën introduced the derived Hall algebras for certain dg-categories in \cite{Toen06}, and the notion was extended by Xiao and Xu in \cite{XX08} for triangulated categories satisfying certain homological finiteness conditions. 
Another important construction is given by Bridgeland in \cite{Bridgeland13}, where he introduced the so-called Bridgeland's Hall algebra by localizing the Hall algebra of two-periodic projective complexes at the classes of acyclic complexes, and used its reduced quotient to realize the whole quantum group. 
Inspired by Bridgeland's construction, semi-derived Hall algebras were introduced by Gorsky in \cite{Gorsky13,Gorsky18} and further developed by Lu and Peng in \cite{LP21}. 
More recently, Hall algebra constructions for root categories were studied by Chen, Lu and Ruan in \cite{CLR25} and Zhang in \cite{Zh25}. 

\subsubsection{Quivers with loops}\;

If $ Q $ has loops, let $ g_i $ be the number of loops at vertex $ i \in \mathbb{I} $ and $ c_{ij} $ be the number of arrows in $ \mathbb{H} $ from $ i $ to $ j $ if $ i \neq j $, then one can associate $
 Q $ with a symmetric Borcherds-Cartan matrix $ A_Q= (a_{ij})_{i,j \in \mathbb{I}} $, where
\begin{align*}
    a_{ij} = \begin{cases}
        2 - 2 g_i &\text{ if } i = j,\\
        -c_{ij}-c_{ji} &\text{ if } i \neq j,
    \end{cases}
\end{align*}
which yields a generalized Kac-Moody algebra in the sense of \cite{Borcherds88} (also called Borcherds algebra), or more generally, yields a Borcherds-Bozec algebra, with infinitely many generators for each imaginary index $ i \in \mathbb{I}^{\mathrm{im}} $. 

As a generalization of the quantum group, the quantum generalized Kac-Moody algebra and its representation theory were introduced in \cite{Kang95}. 
Later, Kang and Schiffmann constructed the canonical basis for its negative half in \cite{KangSchiffmann06} using certain semisimple perverse sheaves, recovering the negative half of the corresponding enveloping algebra after specialization at $ q=1 $.

The quantum Borcherds-Bozec algebra arises from Bozec's works of perverse sheaves on the representation varieties of quivers with loops in \cite{Bozec15,Bozec16}, where he developed the corresponding canonical and crystal basis theories. 
As a further generalization of the quantum generalized Kac-Moody algebra, its positive half can be realized via a Ringel-Hall approach as the generic composition algebra of $ Q $, see \cite{Kang18}.
Later, Lu recovered the whole quantum Borcherds-Bozec algebra in \cite{L23} using the semi-derived Ringel-Hall algebra of nilpotent representations of $ Q $ over $ \mathbb{F}_q $. 
Based on this work, Lin, Lu and Ruan deduced the braid group actions of the quantum Borcherds-Bozec algebra introduced by Fan and Tong in \cite{FT23}, see \cite{LLR23}.

\subsection{Main result}\;

Let $ Q=(\mathbb{I},\mathbb{H},s,t) $ be a finite quiver, possibly having loops, and $ \mathcal{G}_Q $ be the associated Borcherds-Bozec algebra. 
Let $ \mathcal{A}= \mathrm{rep}^{\mathrm{nil}}_{\mathbb{C}}(Q) $ be the category of finite-dimensional nilpotent representations of $ Q $ over $ \mathbb{C} $. 

The purpose of this paper is to provide a geometric realization of the universal enveloping algebra $ U(\mathcal{G}_Q) $ of $ \mathcal{G}_Q $. 

First, we construct a motivic form of the semi-derived Ringel-Hall algebra over $ \mathbb{C}(t) $ following \cite{LP21} in the framework of \cite{Joyce07}, denoted by $ \mathcal{SDH}_t^{\mathrm{red}}(\mathcal{A}) $. 
We note that this construction and the corresponding realization of quantum Borcherds-Bozec algebras can be viewed as a motivic analogue of Lu's construction in \cite{L23}. 
This motivic construction records the relevant geometric information of filtrations. 

Then, we define the classical limit of $ \mathcal{SDH}_t^{\mathrm{red}}(\mathcal{A}) $ at $ t=-1 $, see Definition \ref{Classical limit of SDH}. 
This process is strongly motivated by the recent work of Fang, Lan and Xiao in \cite{FLX24} and by the canonical process of taking the classical limit of quantum Borcherds-Bozec algebras in \cite{FKKB21}. 

Finally, inspired by the injective homomorphism at the generalized quantum group level, see Theorem \ref{Realization of quantum BB algebra}, we define a $ \mathbb{C} $-algebra $ \mathcal{SU}_{-1}(\mathcal{A}) $, see Definition \ref{Definition of SU-1}, which is proved to be isomorphic to $ U(\mathcal{G}_Q) $ in Theorem \ref{SPhi to SU-1}. 
Moreover, we construct a $ \mathbb{C} $-algebra $ \mathcal{SDH}_{-1}^{\mathrm{ind}}(\mathcal{A}) $ which is generated by certain elements corresponding to indecomposable objects and certain Cartan elements, see Definition \ref{Definition of SDHind-1} for details. 

Our main result is the following, see also Corollary \ref{SPhi to SDHind-1}. 

\begin{theorem}
    \label{First main result}
    There is an injective homomorphism of $ \mathbb{C} $-algebras
    \[ \mathcal{S}\Phi: U(\mathcal{G}_Q) \hookrightarrow \mathcal{SDH}^{\mathrm{ind}}_{-1}(\mathcal{A}). \] 
\end{theorem}

\subsection{Structure of the paper}\;

In Section \ref{Motivic form of semi-derived Ringel-Hall algebras}, we construct the motivic semi-derived Ringel-Hall algebra $ \mathcal{SDH}_t^{\mathrm{red}}(\mathcal{A}) $ over $ \mathbb{C}(t) $.

Section \ref{Realization of enveloping algebras of Borcherds-Bozec algebras} is divided into the following three parts. 

In Section \ref{Borcherds-Bozec algebras and Quantum Borcherds-Bozec algebras}, we give a geometric realization of the quantum Borcherds-Bozec algebra using $ \mathcal{SDH}_t^{\mathrm{red}}(\mathcal{A}) $. 
In Section \ref{Classical limit of SDHt}, we give the definition of the classical limit $ \mathcal{SDH}_{-1}^{\mathrm{red}}(\mathcal{A}) $ of  $ \mathcal{SDH}_t^{\mathrm{red}}(\mathcal{A}) $ at $ t=-1 $. 
In Section \ref{Enveloping algebra arising from C2A}, we define two $ \mathbb{C} $-algebras inside $ \mathcal{SDH}_{-1}^{\mathrm{red}}(\mathcal{A}) $ denoted by $ \mathcal{SU}_{-1}(\mathcal{A}) $ and $ \mathcal{SDH}_{-1}^{\mathrm{ind}}(\mathcal{A}) $, respectively.  
We prove that $ U(\mathcal{G}_Q) \cong \mathcal{SU}_{-1}(\mathcal{A})$ and $ \mathcal{SDH}_{-1}^{\mathrm{red}}(\mathcal{A}) \cong \mathcal{SDH}_{-1}^{\mathrm{ind}}(\mathcal{A}) $, respectively. 
Combining these two isomorphisms yields our main result.

\subsection{Notations}\;

For convenience, we list the following notations used throughout the paper. 

\begin{center}
\begin{tabular}{@{}p{0.31\textwidth}p{0.65\textwidth}@{}}
\hline
Notation & Meaning \\
\hline
$\mathcal{A}$ & the category $\mathrm{rep}^{\mathrm{nil}}_{\mathbb{C}}(Q)$ \\
$\mathcal{C}_2(\mathcal{A})$ & the category of $\mathbb{Z}/2$-graded complexes of $\mathcal{A}$ \\
$\mathcal{M},\mathcal{M}^{(2)}$ & the moduli stacks of objects and filtrations in $\mathcal{C}_2(\mathcal{A})$ \\
$ \mathcal{H}_t^{\mathrm{tw}}(\mathcal{C}_2(\mathcal{A}))  $ & the twisted motivic Hall algebra of $\mathcal{C}_2(\mathcal{A})$\\
$ \mathcal{H}_t^{\mathrm{tw}}(\mathcal{C}_2(\mathcal{A}))/(I_2(\mathcal{A})+J_2(\mathcal{A})) $ & the quotient algebra \\
$\mathcal{SDH}_t(\mathcal{A})$ & the localization of the above quotient algebra at classes of acyclic complexes \\
$\mathcal{SDH}_t^{\mathrm{red}}(\mathcal{A})$ & the (reduced) motivic semi-derived Hall algebra \\
$ \mathcal{R},\mathcal{R}^{(2)} $ & the moduli stacks of objects and filtrations in $ \mathcal{A} $\\
$ \mathcal{H}_t^{\mathrm{tw}}(\mathcal{A})  $ & the twisted motivic Hall algebra of $\mathcal{A}$\\
\hline
\end{tabular}
\end{center}

\subsection{Acknowledgment}\;

Fan Xu was supported by the National Natural Science Foundation of China [Grant No. 12371036].

\newpage

\section{Motivic form of semi-derived Ringel-Hall algebras}
\label{Motivic form of semi-derived Ringel-Hall algebras}

Let $ Q = (\mathbb{I},\mathbb{H},s,t)$ be a finite quiver, $ \mathcal{A}= \mathrm{rep}^{\mathrm{nil}}_{\mathbb{C}}(Q) $ be the category of finite-dimensional nilpotent representations of $ Q $ over $ \mathbb{C} $, then $ \mathcal{A} $ is hereditary and abelian. 
If $ Q $ has loops, $ \mathcal{A} $ does not have enough projective objects.

\subsection{Moduli stacks of objects and filtrations in \texorpdfstring{$ \mathcal{C}_2(\mathcal{A}) $}{C2A}}\;
\label{Moduli stacks of objects and filtrations in C2A}

For any dimension vector $ \alpha \in \mathbb{N}\mathbb{I} $, consider a $ \mathbb{C} $-affine space
\begin{align*}
    E_{\alpha}(Q) = \bigoplus_{h\in \mathbb{H}}\mathrm{Hom}_{\mathbb{C}}(\mathbb{C}^{\alpha_{s(h)}},\mathbb{C}^{\alpha_{t(h)}}).
\end{align*}
Any point $ x=(x_{h})_{h \in \mathbb{H}} $ in $ E_{\alpha}(Q) $ defines a representation $  M(x) = (\mathbb{C}^{\alpha}, x) $ of $ Q $ in the fixed vector space $ \mathbb{C}^{\alpha} = \bigoplus_{i \in \mathbb{I}}\mathbb{C}^{{\alpha}_i} $. 

Given a point $ x \in E_{\alpha}(Q) $ and a path $ \rho=h_1 \cdots h_n $ in $ Q $, set $ x_{\rho} = x_{h_1} \cdots x_{h_n} $, then $ x $ is called nilpotent if there exists $ N \geqslant 2 $ such that for any $ n > N $, $ x_{\rho} = 0 $ for any path $ \rho = h_1 \cdots h_n $ in $ Q $  of length $ n $.
Define a closed subvariety
\begin{align*}
     E_{\alpha}^{\mathrm{nil}}(Q) = \{x\in E_{\alpha}(Q)\mid x \text{ is nilpotent} \}.
\end{align*}

Let
\begin{align*}
    G_{\alpha}(Q) = \prod_{i \in \mathbb{I}} \mathrm{GL}_{\alpha_i}(\mathbb{C})
\end{align*}
be the algebraic group which acts on $ E_{\alpha}(Q) $ by $ (g_i)_{i \in \mathbb{I}} \cdot (x_{h})_{h \in \mathbb{H}} = (g_{t(h)} x_{h} g_{s(h)}^{-1})_{h \in \mathbb{H}} $. 
This induces an action of $ G_{\alpha}(Q) $ on $ E_{\alpha}^{\mathrm{nil}}(Q) $, with the quotient stack denoted by $ [E_{\alpha}^{\mathrm{nil}}(Q)/G_{\alpha}(Q)] $.

There is a canonical bijection
\begin{align*}
\left\{
  G_{\alpha}(Q)\text{-orbits in }E_{\alpha}^{\mathrm{nil}}(Q)
\right\}
\;&\longleftrightarrow\;
    \left\{
  \begin{tabular}{c}
      isomorphism classes of   \\
      nilpotent $\mathbb{C}$-representations of $ Q $ \\
      of dimension vector $ \alpha $
  \end{tabular}
\right\}\\
\mathcal{O}_x &\longmapsto [M(x)].
\end{align*}
Moreover, the stabilizer $ \mathrm{Stab}_{G_{\alpha}(Q)}(x) \cong \mathrm{Aut}_{\mathcal{A}} (M(x)) $ is special and affine.

Let $ \mathcal{C}_2(\mathcal{A}) $ be the category of $ \mathbb{Z}/2 $-graded complexes of $ \mathcal{A} $. 
Any object in $ \mathcal{C}_2(\mathcal{A}) $ is of the form
\[\begin{tikzcd}
    X=(X^1,X^0,d^1,d^0) = (X^1 \arrow[r,"d^1"] & X^0 \arrow[l, shift left,"d^0"]).
\end{tikzcd}\]
For any dimension vector pair 
$ \underline{\alpha} = (\alpha^1, \alpha^0) \in \mathbb{N}\mathbb{I} \times \mathbb{N}\mathbb{I} $, similarly as in \cite[Section 3.2]{FLX24}, we define a closed subvariety
\begin{align*}
    C_2(Q, \underline{\alpha}) =& \{ (x^1, x^0, d^1, d^0) \in E_{\alpha^1}^{\mathrm{nil}}(Q) \times E_{\alpha^0}^{\mathrm{nil}}(Q) \times \mathrm{Hom}(\mathbb{C}^{\alpha^1}, \mathbb{C}^{\alpha^0}) 
     \times \mathrm{Hom}(\mathbb{C}^{\alpha^0}, \mathbb{C}^{\alpha^1}) \mid \\
      &d^1_{t(h)} x_h^1 = x_h^0 d^1_{s(h)}, \; d^0_{t(h)} x_h^0 = x_h^1 d^0_{s(h)}, \; d_i^1 d_i^0 = 0, \; d_i^0 d_i^1 = 0, \; \forall h \in \mathbb{H}, i \in \mathbb{I} \},
\end{align*}
where $ \mathrm{Hom}(\mathbb{C}^{\alpha^j}, \mathbb{C}^{\alpha^{j+1}}) = \bigoplus_{i \in \mathbb{I}} \mathrm{Hom}_{\mathbb{C}}(\mathbb{C}^{\alpha_i^j}, \mathbb{C}^{\alpha_i^{j+1}})$ for any $  j \in \mathbb{Z}/2 $.

Let
\[
G_{\underline{\alpha}} = G_{\alpha^1}(Q) \times G_{\alpha^0}(Q)
\]
be the algebraic group which acts on $C_2(Q, \underline{\alpha})$ by
\[(g^1, g^0) \cdot (x^1, x^0, d^1, d^0) 
= (g^1 x^1 (g^1)^{-1}, g^0 x^0 (g^0)^{-1}, g^0 d^1 (g^1)^{-1}, g^1 d^0 (g^0)^{-1}), \]
with the quotient stack denoted by $ [C_2(Q,\underline{\alpha})/G_{\underline{\alpha}}] $.

There is a canonical bijection
\begin{align*}
\left\{
  G_{\underline{\alpha}}\text{-orbits in } C_2(Q,\underline{\alpha})
\right\}
\;&\longleftrightarrow\;
    \left\{
  \begin{tabular}{c}
      isomorphism classes of   \\
      $\mathbb{Z}/2$-graded complexes in $\mathcal{C}_2(\mathcal{A})$ \\
      of dimension vector pair $ \underline{\alpha}=(\alpha^1,\alpha^0) $
  \end{tabular}
\right\}\\
\mathcal{O}_{(x^1, x^0, d^1, d^0)} &\longmapsto [(M(x^1),M(x^0),d^1,d^0)].
\end{align*}
Moreover, the stabilizer $ \mathrm{Stab}_{G_{\underline{\alpha}}}((x^1, x^0, d^1, d^0)) \cong \mathrm{Aut}_{\mathcal{C}_2(\mathcal{A})} ((M(x^1),M(x^0),d^1,d^0)) $ is special and affine.

Finally, we define
\[ C_2^{\mathrm{ind}}(Q,\underline{\alpha}) =\{(x^1,x^0,d^1,d^0) \in C_2(Q,\underline{\alpha}) \mid (M(x^1),M(x^0),d^1,d^0) \text{ is indecomposable} \}. \]
We claim that $ C_2^{\mathrm{ind}}(Q,\underline{\alpha}) \subseteq C_2(Q,\underline{\alpha}) $ is constructible. 
Indeed, for any $ \underline{\alpha} = \underline{\alpha}'+\underline{\alpha}''$, there is an algebraic map
\begin{align}
    \label{psiaa'a''}
    \psi_{\underline{\alpha}',\underline{\alpha}''}^{\underline{\alpha}} : G_{\underline{\alpha}} \times  C_2(Q,\underline{\alpha}') \times   C_2(Q,\underline{\alpha}'') &\to  C_2(Q,\underline{\alpha}) \\
    (g,x',x'')&\mapsto g \cdot (x' \oplus x''), \notag
\end{align}
where $ x'= (x'^1,x'^0,d'^1,d'^0), x''= (x''^1,x''^0,d''^1,d''^0) $, and $ x' \oplus x'' $ denotes the point corresponding to the complex 
\[(M(x'^1) \oplus M(x''^1),M(x'^0) \oplus M(x''^0),\begin{pmatrix}
        \begin{smallmatrix}
            d'^1 & 0 \\
            0 & d''^1
        \end{smallmatrix}
    \end{pmatrix},\begin{pmatrix}
        \begin{smallmatrix}
            d'^0 & 0 \\
            0 & d''^0
        \end{smallmatrix}
         \end{pmatrix}),\]
then each image $ \mathrm{Im}\;\psi_{\underline{\alpha}',\underline{\alpha}''}^{\underline{\alpha}} $ is constructible and thus
\[C_2^{\mathrm{ind}}(Q,\underline{\alpha}) = C_2(Q,\underline{\alpha}) \setminus(\bigcup_{\substack{\underline{\alpha}= \underline{\alpha}'+\underline{\alpha}''\\\text{ s.t. } \underline{\alpha}',\underline{\alpha}'' \neq 0 }   
}\mathrm{Im}\;\psi_{\underline{\alpha}',\underline{\alpha}''}^{\underline{\alpha}})\]
is constructible as the complement of a finite union of constructible subsets. 
Also, it is invariant under the action of $ G_{\underline{\alpha}} $, with the quotient stack denoted by $ [C_2^{\mathrm{ind}}(Q,\underline{\alpha})/G_{\underline{\alpha}}] $.

Although $ \mathcal{A} $ may not have enough projective objects, the following definitions of moduli stacks are entirely analogous to \cite[Section 4.1, 4.2]{FLX24}.

For any $\underline{\alpha} \in \mathbb{N}\mathbb{I} \times \mathbb{N}\mathbb{I} $, let
\[\mathcal{M}_{\underline{\alpha}} = [C_2(Q,\underline{\alpha})/G_{\underline{\alpha}}] 
\]
be the quotient stack parameterizing the isomorphism classes of $\mathbb{Z}/2 $-graded complexes of dimension vector pair $ \underline{\alpha} $. 
Set
\[
\mathcal{M} = \bigsqcup_{\underline{\alpha} \in \mathbb{N}\mathbb{I} \times \mathbb{N}\mathbb{I} } \mathcal{M}_{\underline{\alpha}},
\]
then $ \mathcal{M} $ is an Artin stack, locally of finite type over $ \mathbb{C} $ with affine stabilizers.

Let $ \mathcal{M}(\mathbb{C}) $ be the set of $ \mathbb{C} $-valued points of $ \mathcal{M} $, then there is a bijection
\begin{align*}
    \left\{
  \begin{tabular}{c}
      isomorphism classes of   \\
      $\mathbb{Z}/2$-graded complexes in $\mathcal{C}_2(\mathcal{A})$
  \end{tabular}
\right\}
\;&\longleftrightarrow\;
\left\{
\begin{tabular}{c}
      2-isomorphism classes of   \\
      1-morphisms $ \mathrm{Spec}\;\mathbb{C} \to \mathcal{M} $
  \end{tabular}
\right\} = \mathcal{M}(\mathbb{C})
\\
[X] &\longmapsto \delta_{[X]} = [\mathrm{Spec}\;\mathbb{C} \xrightarrow{i_{[X]}} \mathcal{M}].
\end{align*}
Under the above bijection, we may sometimes identify the isomorphism classes and the corresponding $ \mathbb{C} $-valued points in the following discussion.

For simplicity, write
\[x=(x^1,x^0,d^1,d^0),\;x'=(x'^1,x'^0,d'^1,d'^0),\; x''=(x''^1,x''^0,d''^1,d''^0)\]
for the corresponding points in $ C_2(Q,\underline{\alpha}),C_2(Q,\underline{\alpha}'),C_2(Q,\underline{\alpha}'') $ for any $ \underline{\alpha},\underline{\alpha}',\underline{\alpha}'' \in \mathbb{N}\mathbb{I} \times \mathbb{N}\mathbb{I} $, respectively.

For any $ \underline{\alpha},\underline{\alpha}' \in \mathbb{N}\mathbb{I} \times \mathbb{N}\mathbb{I}$, define a closed subvariety
\begin{align*}
    \mathbf{Hom}(\underline{\alpha},\underline{\alpha}') =\{&(x,x',f^1,f^0)  \in C_2(Q,\underline{\alpha}) \times C_2(Q, \underline{\alpha}') \times \mathrm{Hom}(\mathbb{C}^{\alpha^1},\mathbb{C}^{{\alpha}'^{1}}) \times \mathrm{Hom}(\mathbb{C}^{\alpha^0},\mathbb{C}^{{\alpha}'^{0}})\\
    &\text{ s.t. } f_{t(h)}^j x_h^j = {x'}_h^j f_{s(h)}^j, \;d_i'^j f_i^j = f_i^{j+1} d_i^j  ,\; \forall i \in \mathbb{I}, h\in \mathbb{H}, j\in \mathbb{Z}/2\}.
\end{align*} 
For $ \underline{\alpha}=\underline{\alpha}' +\underline{\alpha}'' \in \mathbb{N}\mathbb{I} \times \mathbb{N}\mathbb{I}$, define a constructible subset
\begin{align*}   W_{\underline{\alpha}',\underline{\alpha}''}^{\underline{\alpha}} = &\{((x'',x,f''^1,f''^0), (x,x',f'^1,f'^0)) \in \mathbf{Hom}(\underline{\alpha}'',\underline{\alpha}) \times \mathbf{Hom}(\underline{\alpha},\underline{\alpha}')\\
    &\text{ s.t. } 0 \to X'' \xrightarrow{(f''^1,f''^0)} X \xrightarrow{(f'^1,f'^0)} X' \to 0 \text{ is exact in } \mathcal{C}_2(\mathcal{A})\},
\end{align*}
where $ X= (M(x^1), M(x^0), d^1, d^0)$ and $ X',X'' $ are defined similarly. 
There is a free action of $ G_{\underline{\alpha}'} \times G_{\underline{\alpha}''}$ on $ W_{\underline{\alpha}',\underline{\alpha}''}^{\underline{\alpha}}  $, with the quotient stack denoted by
\[
{\mathcal{M}^{(2)}}_{\underline{\alpha}',\underline{\alpha}''}^{\underline{\alpha}} = [W_{\underline{\alpha}',\underline{\alpha}''}^{\underline{\alpha}}/G_{\underline{\alpha}'} \times G_{\underline{\alpha}''}].
\]
Set
\[
\mathcal{M}^{(2)} =\bigsqcup _{\underline{\alpha}=\underline{\alpha}' +\underline{\alpha}''}{\mathcal{M}^{(2)}}_{\underline{\alpha}',\underline{\alpha}''}^{\underline{\alpha}},
\]
then $ \mathcal{M}^{(2)} $ parametrizes filtrations of objects in $ \mathcal{C}_2(\mathcal{A}) $ by
\[
((x'',x,f''^1,f''^0), (x,x',f'^1,f'^0)) \mapsto (\mathrm{Im}\;(f''^1,f''^0) \subseteq X).
\]

Consider the diagram
\[
\begin{tikzcd}
\mathcal{M}^{(2)} \arrow[r,"b"] \arrow[d,"{(a_1,a_2)}"] & \mathcal{M} \\
\mathcal{M} \times \mathcal{M},
\end{tikzcd}
\]
where $ (a_1,a_2) $ corresponds to the map $ (X'' \subseteq X) \mapsto (X/X'',X'') $ and $ b $ corresponds to the map $ (X'' \subseteq X) \mapsto X $.

\begin{lemma}[{\itshape\cite[Lemma 4.1]{FLX24}}]\;

    \begin{enumerate}
        \item $ b $ is representable and proper.
        \item $ (a_1, a_2) $ is of finite type.
    \end{enumerate}
\end{lemma}

\begin{corollary}[{\itshape\cite[Corollary 4.2]{FLX24}}]\;

    $ \mathcal{M}^{(2)} $ is an algebraic stack.
\end{corollary}

\subsection{Motivic Hall algebra of \texorpdfstring{$ \mathcal{C}_2(\mathcal{A}) $}{C2A}}\;
\label{Motivic Hall algebra of C2A}

We follow the canonical process of constructing motivic Hall algebras in the framework of \cite{Joyce07}, see also \cite[Section 4.3, 4.4]{FLX24}.

Let $ \mathrm{St}/\mathcal{M} $ be the full subcategory of $ \mathcal{M} $ consisting of objects of the form $ f:\mathcal{F} \to \mathcal{M} $, where $ \mathcal{F} $ is of finite type over $ \mathbb{C} $.

Let $  K(\mathrm{St}/\mathcal{M}) $ be the relative Grothendieck group of $ \mathrm{St}/\mathcal{M} $, see \cite[Section 3.4]{Bridgeland12} or \cite[Section 4.3]{FLX24} for details, and 
$  K^0(\mathrm{St}/\mathcal{M}) $ be its subgroup generated by elements  of the form $ [\mathcal{F} \xrightarrow{f} \mathcal{M}] $, where $ f $ is representable.

In particular, view $ \mathbb{C} $ as a stack, we denote the elements of $  K(\mathrm{St}/\mathbb{C}) $ by $ [\mathcal{F}] = [\mathcal{F} \to \mathbb{C}] $.
There is a unique ring homomorphism $ \Upsilon: K(\mathrm{St}/\mathbb{C}) \to \mathbb{C}(t) $ such that for any smooth projective variety $ X $ over $ \mathbb{C} $, regarded as a representable stack, 
\[\Upsilon([X]) = \sum_{i=0}^{2\mathrm{dim}X} \mathrm{dim}\; \mathrm{H}^i_{\mathrm{c}}(X_{\mathrm{an}},\mathbb{C}) \cdot t^i \in \mathbb{C}[t]\]
is the Poincaré polynomial of $ X $, where $ X_{\mathrm{an}} = X $ is regarded as a compact complex manifold with analytic topology. 
For any projective variety over $ \mathbb{C} $ together with a special algebraic group $ G $-action,
\[\Upsilon([X/G])=\Upsilon([X])\Upsilon([G])^{-1}\] 
by \cite[Example 2.12, Theorem 2.14]{Joyce07} and \cite[Section 5.3]{Bridgeland17}.

\begin{definition}[{\cite[Definition 2.15]{Joyce07}, \cite[Section 5.3]{Bridgeland17}, \cite[Definition 4.3]{FLX24}}]\;
    \label{Scissor relation}    
    
    Let $ K_{\Upsilon}(\mathrm{St}/\mathcal{M}) = \mathbb{C}(t) \otimes_{K(\mathrm{St}/\mathbb{C})} K(\mathrm{St}/\mathcal{M}) $ be the $ \mathbb{C}(t) $-vector space with a basis given by the isomorphism classes of objects in $ \mathrm{St}/\mathcal{M} $ modulo the relations
    \begin{enumerate}
        \item $ [\mathcal{F} \xrightarrow{f} \mathcal{M}] = [\mathcal{V} \xrightarrow{f|_{\mathcal{V}}} \mathcal{M}] + [\mathcal{I} \xrightarrow{f|_{\mathcal{I}}} \mathcal{M}]$ for any object $ \mathcal{F} \xrightarrow{f} \mathcal{M} $ and any closed substack $ \mathcal{V} \subseteq \mathcal{F} $  with open complement $ \mathcal{I} $;
        \item $ [\mathcal{G}_1 \xrightarrow{f g_1}\mathcal{M}] =  [\mathcal{G}_2\xrightarrow{f g_2}\mathcal{M}]  $ for any object $ \mathcal{F} \xrightarrow{f} \mathcal{M} $ and any morphisms $ g_1:\mathcal{G}_1 \to \mathcal{F}, g_2:\mathcal{G}_2 \to \mathcal{F} $ which are Zariski locally trivial fibrations with the same fibers;
        \item $ \Upsilon([\mathcal{G}]) [\mathcal{F} \xrightarrow{f} \mathcal{M}] = [\mathcal{G} \times \mathcal{F} \xrightarrow{f \pi_2} \mathcal{M}]$ for any object $ \mathcal{F} \xrightarrow{f} \mathcal{M} $ in $ K(\mathrm{St}/\mathcal{M}) $ and any oject $ \mathcal{G} $ in $ K(\mathrm{St}/\mathbb{C}) $.
    \end{enumerate}
Let $ K^0_{\Upsilon}(\mathrm{St}/\mathcal{M})= \mathbb{C}(t) \otimes_{K(\mathrm{St}/\mathbb{C})} K^0(\mathrm{St}/\mathcal{M})  $ be the $ \mathbb{C}(t) $-subspace spanned by elements  of the form $ [\mathcal{F} \xrightarrow{f} \mathcal{M}] $, where $ f $ is representable.
\end{definition}

The ring structure of $ K(\mathrm{St}/\mathcal{M}) $ is given by the convolution product
\[ 
b_{*}\circ(a_1,a_2)^*:K(\mathrm{St}/\mathcal{M}) \otimes K(\mathrm{St}/\mathcal{M}) \to K(\mathrm{St}/\mathcal{M}).
\] 
Explicitly, this product, denoted by $ \diamond $, is given by the rule
\[
[\mathcal{F}_1 \xrightarrow{f_1} \mathcal{M}] \diamond [\mathcal{F}_2 \xrightarrow{f_2} \mathcal{M}] = [\mathcal{G} \xrightarrow{bg} \mathcal{M}], 
\]
where $ g:\mathcal{G} \to \mathcal{M}^{(2)} $ is given by the Cartesian diagram
\[
\begin{tikzcd}
\mathcal{G} \arrow[r,"g"]\arrow[d,""] & \mathcal{M}^{(2)} \arrow[r,"b"] \arrow[d,"{(a_1,a_2)}"] & \mathcal{M} \\
\mathcal{F}_1 \times \mathcal{F}_2 \arrow[r,"f_1 \times f_2"]& \mathcal{M} \times \mathcal{M}.
\end{tikzcd}
\]

\begin{theorem}[{\itshape\cite[Theorem 5.2]{Joyce07}, \cite[Theorem 4.3]{Bridgeland12}, \cite[Theorem 4.8, Corollary 4.9]{FLX24}}]\;

    The convolution product $ \diamond $ defines a multiplication such that $ K(\mathrm{St/\mathcal{M}}) $ is an associative algebra over $ K(\mathrm{St}/\mathbb{C}) $ with the unit element 
    \[ 1= [\mathcal{M}_0 \hookrightarrow \mathcal{M}],\]
    where $ \mathcal{M}_0 \cong \mathrm{Spec}\;\mathbb{C} $ is the substack of the zero complex in $ \mathcal{C}_2(\mathcal{A}) $, and $ K^0(\mathrm{St}/\mathcal{M})  $ is a $  K(\mathrm{St}/\mathbb{C}) $-subalgebra. 
    
    Moreover, $ K_{\Upsilon}(\mathrm{St}/\mathcal{M}) $ becomes an associative algebra over $ \mathbb{C}(t) $ under this convolution product with the unit element $ 1 $, and $ K^0_{\Upsilon}(\mathrm{St}/\mathcal{M}) $ is a $ \mathbb{C}(t) $-subalgebra. 
\end{theorem}

Set
\[
\mathcal{H}_t(\mathcal{C}_2(\mathcal{A})) =  K_{\Upsilon}^0(\mathrm{St}/\mathcal{M}),
\]
then by construction
\[
\mathcal{H}_t(\mathcal{C}_2(\mathcal{A})) =  \bigoplus_{\underline{\alpha}\in \mathbb{N}\mathbb{I}\times \mathbb{N}\mathbb{I}}\mathcal{H}_t(\mathcal{C}_2(\mathcal{A}))_{\underline{\alpha}}
\]
is a $ \mathbb{N}\mathbb{I}\times \mathbb{N}\mathbb{I} $-graded $ \mathbb{C}(t) $-algebra, where  $\mathcal{H}_t(\mathcal{C}_2(\mathcal{A}))_{\underline{\alpha}}=K_{\Upsilon}^0(\mathrm{St}/\mathcal{M}_{\underline{\alpha}})$.

The following proposition gives the explicit multiplication structure in $ K^0(\mathrm{St}/\mathcal{M}) $. 

\begin{proposition}[{\itshape\cite[Proposition 5.15]{Joyce07}, \cite[Proposition 4.10]{FLX24}}]\;
\label{Motivic Riedtmann-Peng}

For any $ [\mathcal{F} \xrightarrow{f} \mathcal{M}], [\mathcal{G} \xrightarrow{g} \mathcal{M}] \in K^0(\mathrm{St}/\mathcal{M}) $, assume that
\[\mathcal{F} \times \mathcal{G} \xrightarrow{f \times g} \mathcal{V} \hookrightarrow \mathcal{M} \times \mathcal{M},\]
where $ \mathcal{V} $ is a constructible substack of $ \mathcal{M} \times \mathcal{M} $, then there exists a finite decomposition of $ \mathcal{V} $ into quotient stacks
\[\mathcal{V}= \bigsqcup_{m \in M}[V_m/G_m],\]
where $ V_m $ is a quasi-projective variety over $ \mathbb{C} $, $ G_m $ is a special algebraic group acting on $ V_m $, and $ E_m^0, E_m^1 $ are two finite-dimensional representations of $ G_m $ such that if $ v \in V_m $ projects to $ ([X], [Y]) \in [V_m/G_m](\mathbb{C}) \subseteq \mathcal{M}(\mathbb{C}) \times \mathcal{M}(\mathbb{C}) $ corresponding to two isomorphism classes of objects in $ \mathcal{C}_2(\mathcal{A}) $, then
    \begin{align*}
        \mathrm{Stab}_{G_m}(v) \cong \mathrm{Aut}_{\mathcal{C}_2(\mathcal{A})}(X) \times \mathrm{Aut}_{\mathcal{C}_2(\mathcal{A})}(Y),\;E_m^0 \cong \mathrm{Hom}_{\mathcal{C}_2(\mathcal{A})}(X,Y), \; E_m^1 \cong \mathrm{Ext}^1_{\mathcal{C}_2(\mathcal{A})}(X,Y),
    \end{align*}
    and the action of $  \mathrm{Stab}_{G_m}(v) $ on $ E_m^0, E_m^1 $ coincides with that of $ \mathrm{Aut}_{\mathcal{C}_2(\mathcal{A})}(X) \times \mathrm{Aut}_{\mathcal{C}_2(\mathcal{A})}(Y) $ on $ \mathrm{Hom}_{\mathcal{C}_2(\mathcal{A})}(X,Y) $, $ \mathrm{Ext}^1_{\mathcal{C}_2(\mathcal{A})}(X,Y) $, respectively.

   Under these identifications, the fiber of $ (a_1,a_2) $ over $ [V_m/G_m] $ is represented by $ [V_m\times E_m^1/G_m\ltimes E_m^0] $, with the induced morphism
   \[\pi_m: [V_m\times E_m^1/G_m\ltimes E_m^0] \to \mathcal M^{(2)}\]
   corresponding to the map $ (v, \xi) \mapsto (\mathrm{Im}\;a_{\xi}  \subseteq Z_{\xi}) $, where $ 0 \to Y \xrightarrow{a_{\xi}} Z_{\xi} \xrightarrow{b_{\xi}} X \to 0 $ is the short exact sequence determined by $ \xi \in E_m^1 $.

\end{proposition}

\begin{remark}
    \label{Motivic Riedtmann-Peng (XY)}
    In particular, for any $ \delta_{[X]} = [\mathrm{Spec}\;\mathbb{C} \xrightarrow{i_{[X]}} \mathcal{M} ], \delta_{[Y]} = [\mathrm{Spec}\;\mathbb{C} \xrightarrow{i_{[Y]}} \mathcal{M} ]$ with $ X,Y \in \mathcal{C}_2(\mathcal{A}) $, we have
    \[\delta_{[X]} \diamond \delta_{[Y]} = [[\mathrm{Ext}^1_{\mathcal{C}_2(\mathcal{A})}(X,Y)/(\mathrm{Aut}_{\mathcal{C}_2(\mathcal{A})}(X) \times \mathrm{Aut}_{\mathcal{C}_2(\mathcal{A})}(Y)) \ltimes \mathrm{Hom}_{\mathcal{C}_2(\mathcal{A})}(X, Y)]\xrightarrow{\varphi} \mathcal{M}],\]
    where $ \varphi $ corresponds to the map $ \xi \mapsto [Z_{\xi}] $. 
    This can be viewed as a motivic analogue of the Riedtmann-Peng formula for Ringel-Hall algebras over a finite field. 
\end{remark}

\subsection{Localization and reduced quotient}\;
\label{Localization and reduced quotient}

Let $ K(\mathcal{A}) $ be the Grothendieck group of $ \mathcal{A} $ and $ \langle \cdot, \cdot \rangle : K(\mathcal{A}) \times K(\mathcal{A}) \to \mathbb{Z} $ be the Euler form given by
\begin{align*}
    \langle \widehat{M}, \widehat{N} \rangle = \mathrm{dim}_{\mathbb{C}}\mathrm{Hom}_{\mathcal{A}} (M,N) - \mathrm{dim}_{\mathbb{C}}\mathrm{Ext}^1_{\mathcal{A}}(M,N),
\end{align*}
where $ \widehat{M}, \widehat{N} $ denotes the images of $ M, N \in \mathcal{A} $ in $ K(\mathcal{A}) $. 
Let $ K(\mathcal{C}_2(\mathcal{A})) $ be the Grothendieck group of $ \mathcal{C}_2(\mathcal{A}) $. 
Recall that by \cite[Section 4.3]{LP21} there is a component-wise Euler form 
\begin{align*}
    \langle\cdot,\cdot \rangle_{\mathrm{cw}}: K(\mathcal{C}_2(\mathcal{A})) \times K(\mathcal{C}_2(\mathcal{A})) \to \mathbb{Z}
\end{align*}
given by
\begin{align*}
    \langle X, Y \rangle_{\mathrm{cw}} = \langle \widehat{X^1}, \widehat{Y^1} \rangle + \langle \widehat{X^0}, \widehat{Y^0 }\rangle,
\end{align*}
where $ \begin{tikzcd}
    X = (X^1 \arrow[r,"d^1_X"] & X^0 \arrow[l, shift left,"d^0_X"])
\end{tikzcd} , \begin{tikzcd}
    Y = (Y^1 \arrow[r,"d^1_Y"] & Y^0 \arrow[l, shift left,"d^0_Y"])
\end{tikzcd} \in \mathcal{C}_2(\mathcal{A}) $. 

Define the twist form $ \mathcal{H}_t^{\mathrm{tw}}(\mathcal{C}_2(\mathcal{A})) $ to be the $ \mathbb{C}(t) $-algebra with underlying vector space the same as $ \mathcal{H}_t(\mathcal{C}_2(\mathcal{A})) $ and multiplication given by the twisted multiplication
\begin{align*}
    &[\mathcal{F} \xrightarrow{f} \mathcal{M}_{\underline{\alpha}} \hookrightarrow \mathcal{M}] * [\mathcal{G} \xrightarrow{g} \mathcal{M}_{\underline{\beta}} \hookrightarrow \mathcal{M}] \\
    =& (-t)^{\langle \alpha^1,\beta^1 \rangle + \langle \alpha^0,\beta^0 \rangle} [\mathcal{F} \xrightarrow{f} \mathcal{M}_{\underline{\alpha}} \hookrightarrow \mathcal{M}] \diamond [\mathcal{G} \xrightarrow{g} \mathcal{M}_{\underline{\beta}} \hookrightarrow \mathcal{M}],
\end{align*}
where $ \underline{\alpha} = (\alpha^1, \alpha^0), \underline{\beta} = (\beta^1, \beta^0) \in \mathbb{N}\mathbb{I} \times \mathbb{N}\mathbb{I} $. 
$ \mathcal{H}_t^{\mathrm{tw}}(\mathcal{C}_2(\mathcal{A})) $ is still associative and $ \mathbb{N}\mathbb{I} \times \mathbb{N}\mathbb{I} $-graded by the bilinearity of $ \langle \cdot, \cdot \rangle_{\mathrm{cw}} $.

Let $ \mathcal{C}_{2,\mathrm{ac}}(\mathcal{A}) $ be the full subcategory of $ \mathcal{C}_2(\mathcal{A}) $ consisting of acyclic complexes. 
Let $ I_2(\mathcal{A}) $ be the two-sided ideal of $ \mathcal{H}_t^{\mathrm{tw}}(\mathcal{C}_2(\mathcal{A})) $ generated by all differences
\begin{align*}
    \Upsilon(\mathrm{Aut}_{\mathcal{C}_2(\mathcal{A})}(L))\delta_{[L]} - \Upsilon(\mathrm{Aut}_{\mathcal{C}_2(\mathcal{A})}(K \oplus X))\delta_{[K \oplus X]} 
\end{align*}
if there exists a short exact sequence
\begin{align*}
     0 \to K \to L \to X \to 0  \text{ in } \mathcal{C}_2(\mathcal{A})  \text{ with } K \in \mathcal{C}_{2,\mathrm{ac}}(\mathcal{A}).
\end{align*}
Dually, let $ J_2(\mathcal{A}) $ be the two-sided ideal of $ \mathcal{H}_t^{\mathrm{tw}}(\mathcal{C}_2(\mathcal{A})) $ generated by all differences
\begin{align*}
    \Upsilon(\mathrm{Aut}_{\mathcal{C}_2(\mathcal{A})}(L))\delta_{[L]} - \Upsilon(\mathrm{Aut}_{\mathcal{C}_2(\mathcal{A})}(K \oplus X))\delta_{[K \oplus X]} 
\end{align*}
if there exists a short exact sequence
\begin{align*}
     0 \to X \to L \to K \to 0  \text{ in } \mathcal{C}_2(\mathcal{A})  \text{ with } K \in \mathcal{C}_{2,\mathrm{ac}}(\mathcal{A}).
\end{align*}

Consider the quotient algebra
\begin{align*}
    \mathcal{H}_t^{\mathrm{tw}}(\mathcal{C}_2(\mathcal{A}))/(I_2(\mathcal{A})+J_2(\mathcal{A})),
\end{align*}
with induced multiplication still denoted by $ * $ and we use the same symbols in $ \mathcal{H}_t^{\mathrm{tw}}(\mathcal{C}_2(\mathcal{A})) $ and $ \mathcal{H}_t^{\mathrm{tw}}(\mathcal{C}_2(\mathcal{A}))/(I_2(\mathcal{A})+J_2(\mathcal{A})) $.

By \cite[Proposition 2.3]{LP21},  there are well-defined bilinear forms, still called Euler forms, 
\begin{align*}
    \langle \cdot, \cdot \rangle : K(\mathcal{C}_{2,\mathrm{ac}}(\mathcal{A})) \times K(\mathcal{C}_2(\mathcal{A})) \to \mathbb{Z} \text{ and }
     \langle \cdot, \cdot \rangle : K(\mathcal{C}_2(\mathcal{A})) \times K(\mathcal{C}_{2,\mathrm{ac}}(\mathcal{A})) \to \mathbb{Z}
\end{align*}
given by
\begin{align*}
    \langle K,X \rangle &=  \mathrm{dim}_{\mathbb{C}}\mathrm{Hom}_{\mathcal{C}_2(\mathcal{A})}(K,X) - \mathrm{dim}_{\mathbb{C}}\mathrm{Ext}^1_{\mathcal{C}_2(\mathcal{A})}(K,X), \\
    \langle X,K \rangle &=  \mathrm{dim}_{\mathbb{C}}\mathrm{Hom}_{\mathcal{C}_2(\mathcal{A})}(X,K) - \mathrm{dim}_{\mathbb{C}}\mathrm{Ext}^1_{\mathcal{C}_2(\mathcal{A})}(X,K), 
\end{align*}
where $ K \in \mathcal{C}_{2,\mathrm{ac}}(\mathcal{A})  $ and $ X \in \mathcal{C}_2(\mathcal{A}) $. 

For any $ A \in \mathcal{A} $, consider the complexes
\[\begin{tikzcd}
    C_A = (0 \arrow[r,"0"] & A),\;  C_A^* = (A \arrow[l, shift left, "0"] \arrow[r,"0"] & 0) \arrow[l, shift left, "0"],
    K_A = (A \arrow[r,"1"] & A),\;  K_A^* = (A \arrow[l, shift left, "0"] \arrow[r,"0"] & A \arrow[l, shift left, "1"])
\end{tikzcd} \]
in $ \mathcal{C}_2(\mathcal{A}) $.

\begin{proposition}
     \label{KX=XK}
     For any   $ \begin{tikzcd}
    K = (K^1 \arrow[r,"k^1"] & K^0) \arrow[l, shift left,"k^0"] \end{tikzcd} \in \mathcal{C}_{2,\mathrm{ac}}(\mathcal{A}) $ and $ [\mathcal{F} \xrightarrow{f} \mathcal{M}_{\underline{\beta}} \hookrightarrow \mathcal{M} ] \in \mathcal{H}_t^{\mathrm{tw}}(\mathcal{C}_2(\mathcal{A})) $, in $ \mathcal{H}_t^{\mathrm{tw}}(\mathcal{C}_2(\mathcal{A}))/(I_2(\mathcal{A})+J_2(\mathcal{A})) $, we have
     \begin{align*}
         &\delta_{[K]} * [\mathcal{F} \xrightarrow{f} \mathcal{M}_{\underline{\beta}} \hookrightarrow \mathcal{M} ] \\
         =& (-t)^{\langle \widehat{K^1} - 2 \widehat{\mathrm{Im}\;k^1},\beta^1 \rangle + \langle \widehat{K^0} - 2 \widehat{\mathrm{Im}\;k^0},\beta^0 \rangle + \langle \beta^1,2 \widehat{\mathrm{Im}\;k^0}-\widehat{K^1}  \rangle + \langle \beta^0, 2 \widehat{\mathrm{Im}\;k^1}-\widehat{K^0}\rangle}[\mathcal{F} \xrightarrow{f} \mathcal{M}_{\underline{\beta}} \hookrightarrow \mathcal{M} ] * \delta_{[K]}.
     \end{align*}
     
\end{proposition}
\begin{proof}
    By Proposition \ref{Motivic Riedtmann-Peng} and \cite[Corollary 5.16]{Joyce07}, keeping the notations therein we can write
    \begin{align*}
    &\delta_{[K]} \diamond [\mathcal{F} \xrightarrow{f} \mathcal{M}_{\underline{\beta}} \hookrightarrow \mathcal{M} ] \\
    =& \sum_{m \in M}\sum_{n \in N_m}c_{mn}[[W_{mn} \times E_m^1/G_m \ltimes E_m^0] \xrightarrow{\Tilde{\tau}_{mn}}[V_m \times E_m^1/G_m \ltimes E_m^0] \xrightarrow{\varphi_m }\mathcal{M}],
    \end{align*}
    where $ W_{mn} $ is a quasi-projective variety over $ \mathbb{C} $ acted on by $ G_m $, $ c_{mn} \in \mathbb{C}(t) $, $ \Tilde{\tau}_{mn} $ is induced by the induced map $ \Bar{\tau}_{mn}: [W_{mn}/G_m] \to [V_m/G_m] $ of the $ G_m $-equivariant morphism $ \tau_{mn}: W_{mn} \to V_m $, and $ \varphi_m = b \pi_m  $ corresponds to the map $ (v,\xi) \mapsto [Z_{\xi}] $.

   If $ v \in V_m $ projects to $ ([K],[X]) $, then
\begin{align*}
    \mathrm{Stab}_{G_m}(v) \cong \mathrm{Aut}_{\mathcal{C}_2(\mathcal{A})}(K) \times \mathrm{Aut}_{\mathcal{C}_2(\mathcal{A})}(X),\;E_m^1 \cong \mathrm{Ext}^1_{\mathcal{C}_2(\mathcal{A})}(K,X),\; E_m^0 \cong \mathrm{Hom}_{\mathcal{C}_2(\mathcal{A})}(K,X). 
\end{align*}
Let $  \mathrm{Ext}^1_{\mathcal{C}_2(\mathcal{A})}(K, X)_{Z_{\xi}} \subseteq \mathrm{Ext}^1_{\mathcal{C}_2(\mathcal{A})}(K, X) $  be the constructible subset consisting of $ \xi' $  such that the middle term of the short exact sequence representing $ \xi' $ is isomorphic to $ Z_{\xi} $.
In $ \mathcal{H}_t^{\mathrm{tw}}(\mathcal{C}_2(\mathcal{A}))/(I_2(\mathcal{A})+J_2(\mathcal{A})) $, we have
\begin{align*}
    &\Upsilon(\mathrm{Ext}^1_{\mathcal{C}_2(\mathcal{A})}(K, X)_{Z_{\xi}})\Upsilon(\mathrm{Aut}_{\mathcal{C}_2(\mathcal{A})}(Z_{\xi}))\delta_{[Z_{\xi}]} \\
    =& \Upsilon(\mathrm{Ext}^1_{\mathcal{C}_2(\mathcal{A})}(K, X)_{Z_{\xi}})\Upsilon(\mathrm{Aut}_{\mathcal{C}_2(\mathcal{A})}(K \oplus X))[\mathrm{Spec}\;\mathbb{C} \xrightarrow{i_{[K \oplus X]}} \mathcal{M}] \\
    =& \Upsilon(\mathrm{Ext}^1_{\mathcal{C}_2(\mathcal{A})}(K, X)_{Z_{\xi}})[\mathrm{Ext}^1_{\mathcal{C}_2(\mathcal{A})}(K,X)_{K \oplus X} \xrightarrow{\varphi_m\pi} \mathcal{M}],
\end{align*}
where $ \pi $ is the restriction of the natural projection and $ \varphi_m \pi $  corresponds to the map $ 0 \mapsto [K \oplus X] $. 
Then
\begin{align*}
    &[[E_m^1/\mathrm{Stab}_{G_m \ltimes E_m^0}(v)] \xrightarrow{\varphi_m} \mathcal{M}] \\
    =& [[\mathrm{Ext}^1_{\mathcal{C}_2(\mathcal{A})}(K,X) /(\mathrm{Aut}_{\mathcal{C}_2(\mathcal{A})}(K) \times \mathrm{Aut}_{\mathcal{C}_2(\mathcal{A})}(X)) \ltimes \mathrm{Hom}_{\mathcal{C}_2(\mathcal{A})}(K, X)] \xrightarrow{\varphi_m} \mathcal{M}] \\
    =& \Upsilon(\mathrm{Ext}^1_{\mathcal{C}_2(\mathcal{A})}(K,X))\\
    &[[\mathrm{Ext}^1_{\mathcal{C}_2(\mathcal{A})}(K,X)_{K \oplus X}/(\mathrm{Aut}_{\mathcal{C}_2(\mathcal{A})}(K) \times \mathrm{Aut}_{\mathcal{C}_2(\mathcal{A})}(X)) \ltimes \mathrm{Hom}_{\mathcal{C}_2(\mathcal{A})}(K, X)] \xrightarrow{\varphi_m'} \mathcal{M}] \\
    =& \Upsilon(\mathrm{Ext}^1_{\mathcal{C}_2(\mathcal{A})}(K,X))\Upsilon(\mathrm{Hom}_{\mathcal{C}_2(\mathcal{A})}(K, X))^{-1}[[(E_m^1)'/\mathrm{Stab}_{G_m}(v)] \xrightarrow{\varphi_m'} \mathcal{M}] \\
    =&  \Upsilon(\mathrm{Ext}^1_{\mathcal{C}_2(\mathcal{A})}(K,X))\Upsilon(\mathrm{Hom}_{\mathcal{C}_2(\mathcal{A})}(K, X))^{-1}\Upsilon(\mathrm{Stab}_{G_m}(v))^{-1}[(E_m^1)' \xrightarrow{\varphi_m'\pi'} \mathcal{M}],
\end{align*}
where $ \varphi_m' $ is the restriction map of $ \varphi_m $, $ (E_m^1)' \cong \mathrm{Ext}^1_{\mathcal{C}_2(\mathcal{A})}(K,X)_{K \oplus X} =\{0\} $, $ \pi': (E_m^1)' \to[(E_m^1)'/\mathrm{Stab}_{G_m}(v)] $ is the natural projection, and $ \varphi_m'\pi' $ corresponds to the map $ 0 \mapsto [K \oplus X] $.

Let $\begin{tikzcd} X = (X^1 \arrow[r,"d^1"] & X^0) \arrow[l, shift left,"d^0"] \end{tikzcd}$, it follows from the proof of \cite[Proposition 3.12]{LP21} and \cite[Proposition 2.4, Corollary 2.5]{LP21} that
    \begin{align*}
    &\langle \widehat{K^1},\widehat{X^1} \rangle + \langle \widehat{K^0},\widehat{X^0} \rangle - 2 \langle K,X \rangle \\
    =& \langle \widehat{K^1},\widehat{X^1} \rangle + \langle \widehat{K^0},\widehat{X^0} \rangle - 2 \langle K_{\mathrm{Im}\;k^1} \oplus K_{\mathrm{Im}\;k^0}^*,C_{H^0(X)} \oplus C_{H^1(X)}^* \oplus K_{\mathrm{Im}\;d^1} \oplus K_{\mathrm{Im}\;d^0}^* \rangle \\
    =& \langle \widehat{K^1} - 2 \widehat{\mathrm{Im}\;k^1},\widehat{X^1} \rangle + \langle \widehat{K^0} - 2 \widehat{\mathrm{Im}\;k^0},\widehat{X^0} \rangle.
\end{align*}
Thus, using $ \Upsilon(\mathbb{C}^n) = t^{2n} $ and taking the twist into consideration, we have
\begin{align*}
    &\delta_{[K]} * [\mathcal{F} \xrightarrow{f} \mathcal{M}_{\underline{\beta}} \hookrightarrow \mathcal{M} ] \\
    =& \sum_{m \in M}\sum_{n \in N_m}c_{mn}(-t)^{\langle \widehat{K^1} - 2 \widehat{\mathrm{Im}\;k^1},\beta^1 \rangle + \langle \widehat{K^0} - 2 \widehat{\mathrm{Im}\;k^0},\beta^0 \rangle}\Upsilon(G_m)^{-1} [(W_{mn} \times E_m^1)' \xrightarrow{\varphi_m' \Tilde{\tau}_{mn} \pi_{mn}'} \mathcal{M}] ,
\end{align*}
where $ (W_{mn} \times E_m^1)' \subseteq W_{mn} \times E_m^1 $ is a constructible subset consisting of $ (w,\xi) $ such that if $ \tau_{mn} w = v $ projects to $ ([K],[X]) $, then $ Z_{\xi} \cong K \oplus X $, $ \pi_{mn}' $ is the restriction map of the natural projection $ \pi_{mn}:  W_{mn} \times E_m^1 \to [W_{mn} \times E_m^1/G_m \ltimes E_m^0]  $ and $ \varphi_m' \Tilde{\tau}_{mn} \pi_{mn}' $ corresponds to the map $ (w,0) \mapsto [K \oplus X] $.

For the calculation of $ [\mathcal{F} \xrightarrow{f} \mathcal{M}_{\underline{\beta}} \hookrightarrow \mathcal{M} ] * \delta_{[K]}  $, note that $  [\mathrm{Spec}\;\mathbb{C} \times \mathcal{F} \xrightarrow{i_{[K]} \times f } \mathcal{M}] = [\mathcal{F} \times \mathrm{Spec}\;\mathbb{C} \xrightarrow{f \times i_{[K] }} \mathcal{M}  ] $ via the natural isomorphism $ \mathrm{Spec}\;\mathbb{C} \times \mathcal{F} \cong \mathcal{F} \times \mathrm{Spec}\;\mathbb{C} $, then the only difference is the power of $ (-t) $. 
By a similar computation we have
\begin{align*}
    &[\mathcal{F} \xrightarrow{f} \mathcal{M}_{\underline{\beta}} \hookrightarrow \mathcal{M} ] * \delta_{[K]} \\
    =& \sum_{m \in M}\sum_{n \in N_m}c_{mn}(-t)^{\langle \beta^1,\widehat{K^1} - 2 \widehat{\mathrm{Im}\;k^0} \rangle + \langle \beta^0, \widehat{K^0} - 2 \widehat{\mathrm{Im}\;k^1}\rangle}\Upsilon(G_m)^{-1} [(W_{mn} \times E_m^1)' \xrightarrow{\varphi_m' \Tilde{\tau}_{mn} \pi_{mn}'} \mathcal{M}] ,
\end{align*}
which finishes the proof.

\end{proof}

In particular, we have the following.

\begin{corollary}
    \label{KACB}
    For any $ A,B \in \mathcal{A} $, in $ \mathcal{H}_t^{\mathrm{tw}}(\mathcal{C}_2(\mathcal{A}))/(I_2(\mathcal{A})+J_2(\mathcal{A})) $, we have
    \begin{align*}
        \delta_{[K_A]} * \delta_{[C_B]} &= (-t)^{(\widehat{A},\widehat{B})}  \delta_{[C_B]} * \delta_{[K_A]},\quad \delta_{[K_A^*]} * \delta_{[C_B]} = (-t)^{-(\widehat{A},\widehat{B})}  \delta_{[C_B]} * \delta_{[K_A^*]},\\
        \delta_{[K_A]} * \delta_{[C_B^*]} &= (-t)^{-(\widehat{A},\widehat{B})} \delta_{[C_B^*]} * \delta_{[K_A]},\quad \delta_{[K_A^*]} * \delta_{[C_B^*]} = (-t)^{(\widehat{A},\widehat{B})}  \delta_{[C_B^*]} * \delta_{[K_A^*]},
    \end{align*}
    where $ (\cdot, \cdot):K(\mathcal{A}) \times K(\mathcal{A}) \to \mathbb{Z} $ is the symmetric Euler form given by $ (\widehat{M},\widehat{N}) = \langle \widehat{M}, \widehat{N} \rangle + \langle \widehat{N}, \widehat{M} \rangle $.

\end{corollary}

\begin{definition}
    For any $ X \in \mathcal{A} $, define the elements
\begin{align*}
    b_X = \Upsilon(\mathrm{Aut}_{\mathcal{C}_2(\mathcal{A})}(K_X)) \delta_{[K_X]},\;
    b_X^* = \Upsilon(\mathrm{Aut}_{\mathcal{C}_2(\mathcal{A})}(K_X^*)) \delta_{[K_X^*]} \in \mathcal{H}_t^{\mathrm{tw}}(\mathcal{C}_2(\mathcal{A})).
\end{align*}
\end{definition}

\begin{corollary}
    \label{bAbB}
    For any $ A,B \in \mathcal{A} $, in $ \mathcal{H}_t^{\mathrm{tw}}(\mathcal{C}_2(\mathcal{A}))/(I_2(\mathcal{A})+J_2(\mathcal{A})) $, we have
    \begin{align*}
        &b_A * b_B = b_{A \oplus B},\;b_A * b_B^* = \Upsilon(\mathrm{Aut}_{\mathcal{C}_2(\mathcal{A})}(K_A \oplus K_B^*))\delta_{[K_A \oplus K_B^*]},\\
        &[b_A,b_B] = [b_A,b_B^*] = [b_A^*,b_B^*]=0,
    \end{align*}
    where the Lie bracket is given by the commutator $ [a,b] = a * b - b * a $.
\end{corollary}
\begin{proof}
  As a special case of Proposition \ref{KX=XK}, we have
  \begin{align*}
      &\delta_{[K_A]} \diamond \delta_{[K_B]} \\
      =& [[\mathrm{Ext}^1_{\mathcal{C}_2(\mathcal{A})}(K_A,K_B)/(\mathrm{Aut}_{\mathcal{C}_2(\mathcal{A})}(K_A) \times \mathrm{Aut}_{\mathcal{C}_2(\mathcal{A})}(K_B)) \ltimes \mathrm{Hom}_{\mathcal{C}_2(\mathcal{A})}(K_A, K_B)] \xrightarrow{\varphi_m} \mathcal{M}] \\
      =& t^{-2\mathrm{dim}_{\mathbb{C}}\mathrm{Hom}_{\mathcal{A}}(A,B)} \Upsilon(\mathrm{Aut}_{\mathcal{C}_2(\mathcal{A})}(K_A))^{-1}\Upsilon(\mathrm{Aut}_{\mathcal{C}_2(\mathcal{A})}(K_B))^{-1}[\mathrm{Ext}^1_{\mathcal{C}_2(\mathcal{A})}(K_A,K_B) \xrightarrow{\varphi_m \pi} \mathcal{M}]  \\
      =& t^{-2\langle \widehat{A}, \widehat{B} \rangle}\Upsilon(\mathrm{Aut}_{\mathcal{C}_2(\mathcal{A})}(K_A))^{-1}\Upsilon(\mathrm{Aut}_{\mathcal{C}_2(\mathcal{A})}(K_B))^{-1} [\mathrm{Ext}^1_{\mathcal{C}_2(\mathcal{A})}(K_A,K_B)_{K_A \oplus K_B} \xrightarrow{\varphi_m' \pi} \mathcal{M}], 
  \end{align*}
  where $ \varphi_m' \pi $ corresponds to the map $ 0 \mapsto [K_A \oplus K_B]  $. 
  Then
  \begin{align*}
      b_A * b_B = t^{2\langle \widehat{A}, \widehat{B} \rangle-2\langle \widehat{A}, \widehat{B} \rangle}\Upsilon(\mathrm{Aut}_{\mathcal{C}_2(\mathcal{A})}(K_A \oplus K_B))[\mathrm{Spec}\; \mathbb{C} \xrightarrow{i_{[K_A \oplus K_B]}} \mathcal{M}] = b_{ A\oplus B}.
  \end{align*}
\end{proof}

Define a subset
\begin{align*}
    S_2(\mathcal{A}) = \{a\delta_{[K]} \in \mathcal{H}_t^{\mathrm{tw}}(\mathcal{C}_2(\mathcal{A}))/(I_2(\mathcal{A})+J_2(\mathcal{A}))\mid a \in \mathbb{C}(t)^{\times}, K \in \mathcal{C}_{2,\mathrm{ac}}(\mathcal{A})\},
\end{align*}
which is multiplicatively closed with the identity $ \delta_{[0]} = [\mathcal{M}_0 \hookrightarrow \mathcal{M} ] $. 
By Proposition \ref{KX=XK}, we obtain the following.

\begin{proposition}
    $ S_2(\mathcal{A}) $  satisfies the left and right Ore conditions and is left and right reversible. 
    In particular, the right localization of $ \mathcal{H}_t^{\mathrm{tw}}(\mathcal{C}_2(\mathcal{A}))/(I_2(\mathcal{A})+J_2(\mathcal{A}))$ with respect to $ S_2(\mathcal{A}) $ exists.
\end{proposition}

\begin{definition}
    Define the $ \mathbb{C}(t) $-algebra
    \begin{align*}
        \mathcal{SDH}_t(\mathcal{A}) = \mathcal{H}_t^{\mathrm{tw}}(\mathcal{C}_2(\mathcal{A}))/(I_2(\mathcal{A})+J_2(\mathcal{A}))[S_2(\mathcal{A})^{-1}].
    \end{align*}
\end{definition}

\begin{remark}\;
   \label{q = t^2}

    \begin{enumerate}
        \item Here we use a left-right symmetric construction as in \cite[Section 4.2]{CH24} for simplicity. 
    Indeed, by a similar argument as in \cite[Lemma 4.7]{CH24}, we can show that this so-called motivic semi-derived Ringel-Hall algebra will coincide with the one defined in the way given by \cite{LP21}, up to isomorphisms of $ \mathbb{C}(t) $-algebras. 
     \item We also note that the computations in $ \mathcal{SDH}_t^{\mathrm{red}}(\mathcal{A}) $ are sometimes formally analogous to those in the semi-derived Ringel-Hall algebra over a finite field $ \mathbb{F}_q $: finite counts of the form $ q^k $ can be replaced by virtual Poincaré polynomials of constructible sets of the form $ \mathbb{C}^k $, which is just $ t^{2k} $, see \cite[Example 5.21]{Joyce07}.
    \end{enumerate}

\end{remark}

For any $ \alpha \in K(\mathcal{A}) $, there exist $ A,B \in \mathcal{A} $ such that $ \alpha = \widehat{A} - \widehat{B} $, then by Corollary \ref{bAbB}, there are well-defined elements
\begin{align*}
    b_{\alpha} = b_A * b_B^{-1}, \; b_{\alpha}^* = b_A^* * b_B^{*-1} \in \mathcal{SDH}_t(\mathcal{A}).
\end{align*}

\begin{definition}
    Define the reduced quotient of $ \mathcal{SDH}_t(\mathcal{A}) $ to be the $ \mathbb{C}(t) $-algebra
    \begin{align*}
        \mathcal{SDH}_t^{\mathrm{red}}(\mathcal{A}) = \mathcal{SDH}_t(\mathcal{A})/\langle b_{\alpha} * b_{\alpha}^* -1 \mid \alpha \in K(\mathcal{A})\rangle,
    \end{align*}
    with induced multiplication still denoted by $ * $ and we use the same symbols in $ \mathcal{SDH}_t(\mathcal{A})  $ and $ \mathcal{SDH}_t^{\mathrm{red}}(\mathcal{A})  $.
\end{definition}

\begin{corollary}
    \label{balphaalpha'}
    For any $ \alpha,\alpha' \in K(\mathcal{A}) $ and $ [\mathcal{F} \xrightarrow{f} \mathcal{M}_{\underline{\beta}} \hookrightarrow \mathcal{M} ] \in \mathcal{SDH}_t^{\mathrm{red}}(\mathcal{A}) $, we have
    \begin{align*}
        &b_{\alpha} * b_{\alpha'} = b_{\alpha + \alpha'},\; [b_{\alpha},b_{\alpha'}] = 0,\\
         &b_{\alpha} * [\mathcal{F} \xrightarrow{f} \mathcal{M}_{\underline{\beta}} \hookrightarrow \mathcal{M} ] = (-t)^{(\alpha,\beta^0-\beta^1)}[\mathcal{F} \xrightarrow{f} \mathcal{M}_{\underline{\beta}} \hookrightarrow \mathcal{M} ] * b_{\alpha}.
    \end{align*}
\end{corollary}

\subsection{Triangular decomposition}\;

First, we recall two well-known facts and reformulate them in the case of $ \mathcal{C}_2(\mathcal{A}) $.

\begin{lemma}[{\itshape\cite[Lemma 6.9]{FLX24}}]\;
    \label{Upsilon(Aut)}
    
    Let $ X_1, \cdots, X_m \in \mathcal{C}_2(\mathcal{A}) $ be indecomposable complexes with $ n $ many distinct isomorphism classes of sizes $m_1,\cdots,m_n$ such that $m_1+\cdots+m_n=m$. 
    Set $ X =  X_1 \oplus \cdots \oplus X_m $, then
    \begin{align*}
        \Upsilon(\mathrm{Aut}_{\mathcal{C}_2(\mathcal{A})}(X)) &= t^{2 \mathrm{dim}\; \mathrm{rad}\;\mathrm{End}_{\mathcal{C}_2(\mathcal{A})}(X)}\prod_{i=1}^n (t^{m_i(m_i-1)}\prod^{m_i}_{k=1}(t^{2k}-1)) \\
        &= t^{2l_{(X_1,\cdots,X_m)}} \prod_{k=1}^m \Upsilon(\mathrm{Aut}_{\mathcal{C}_2(\mathcal{A})}(X_k)) \prod_{i=1}^n\frac{t^{m_i(m_i-1)}\prod^{m_i}_{k=1}(t^{2k}-1)}{(t^2-1)^{m_i}},
    \end{align*}
    where $ l_{(X_1,\cdots,X_m)} = \mathrm{dim}\;\mathrm{rad}\;\mathrm{End_{\mathcal{C}_2(\mathcal{A})}}(X)-\sum_{i=1}^m\mathrm{dim}\;\mathrm{rad}\;\mathrm{End_{\mathcal{C}_2(\mathcal{A})}}(X_i) $.
    
    In particular, if $ m \geqslant m' $ for some positive integer $ m' $, then $ (t^2-1)^{m'} \mid \Upsilon(\mathrm{Aut}_{\mathcal{C}_2(\mathcal{A})}(X)) $.
    \end{lemma}

The following lemma can be proved analogously to \cite[Lemma 5.2]{DXX10} by showing the existence of a free $ \mathbb{C}^* $-action.

\begin{lemma}
    \label{Ext^1(X,Y)Z}
     For any $ X,Y,Z \in \mathcal{C}_2(\mathcal{A}) $ such that $ \mathrm{Ext}^1_{\mathcal{C}_2(\mathcal{A})}(X, Y)_{Z} \neq \emptyset $, if $ Z \not\cong  X \oplus Y  $, then
    \begin{align*}
     (t^2-1) \mid \Upsilon(\mathrm{Ext}^1_{\mathcal{C}_2(\mathcal{A})}(X,Y)_{Z}).
    \end{align*}
\end{lemma}

Next, we introduce some notions that will be used throughout the paper. 

Let $ \mathcal{O}_i \subseteq  C_2(Q,\underline{\alpha}_i) $ be a $ G_{\underline{\alpha}_i} $-invariant constructible subset for $ i=1,2 $. 
Define
\begin{align*}
    \mathcal{O}_1 \oplus \mathcal{O}_2 &= \mathrm{Im}\;\psi^{\underline{\alpha}_1+\underline{\alpha}_2}_{\underline{\alpha}_1, \underline{\alpha}_2}, \text{ where } \psi^{\underline{\alpha}_1+\underline{\alpha}_2}_{\underline{\alpha}_1, \underline{\alpha}_2} \text{ is given by (\ref{psiaa'a''})}.
\end{align*}
Then $ \mathcal{O}_1 \oplus \mathcal{O}_2 $ is constructible and $ G_{\underline{\alpha}_1+\underline{\alpha}_2} $-invariant with the quotient stack denoted by $ [\mathcal{O}_1 \oplus \mathcal{O}_2/G_{\underline{\alpha}_1+\underline{\alpha}_2} ] $.

Denote by $ n\mathcal{O} $ the direct sum of $ n $ copies of $ \mathcal{O} $ and denote by $ G_{n\underline{\alpha}} $ the algebraic group $ G_{\underline{\alpha}+\cdots+ \underline{\alpha}} $, then $ n\mathcal{O} $ is invariant under $ G_{n\underline{\alpha}} $ by induction on $ n $. 

\begin{definition}
    Let $ \mathcal{O} \subseteq C_2(Q,\underline{\alpha}) $ be a $G_{\underline{\alpha}}$-invariant constructible subset. 
    
    $ \mathcal{O} $ is called indecomposable if all points in it correspond to indecomposable complexes. 
    
    $ \mathcal{O} $ is called of Krull-Schmidt type if $ \mathcal{O} = m_1 \mathcal{O}_1 \oplus \cdots \oplus m_n \mathcal{O}_n $
    for some indecomposable constructible sets $ \mathcal{O}_1, \cdots, \mathcal{O}_n $, which are invariant under the corresponding group actions, and $ m_1, \cdots, m_n \in \mathbb{N} $. 
       
\end{definition}

\begin{remark}
    \label{disjoint union}
    Given a constructible set $ \mathcal{O} = m_1 \mathcal{O}_1 \oplus \cdots \oplus m_n \mathcal{O}_n $ of Krull-Schmidt type, if there exist $ \mathcal{O}_i \neq \mathcal{O}_j $ such that $ \mathcal{O}_i \cap \mathcal{O}_j \neq \emptyset  $, then
\begin{align*}
    \mathcal{O}_i \oplus \mathcal{O}_j = 2 {\mathcal{I}} \bigsqcup ({\mathcal{C}}_i \oplus {\mathcal{C}}_j) \bigsqcup ({\mathcal{C}}_i \oplus {\mathcal{I}}) \bigsqcup ({\mathcal{I}} \oplus {\mathcal{C}}_j)
\end{align*}
can be written as a disjoint union, where the intersection $ \mathcal{I} = \mathcal{O}_i \cap \mathcal{O}_j$ and the corresponding complements $ {\mathcal{C}}_i=\mathcal{O}_i \backslash {\mathcal{I}} $, $ {\mathcal{C}}_j=\mathcal{O}_j \backslash {\mathcal{I}} $ are constructible. 
Proceeding in this way, for finitely many times, we can write
\begin{align*}
    \mathcal{O} = \bigsqcup_{i=1}^l (m_{i,1}\mathcal{O}_{i,1} \oplus m_{i,2}\mathcal{O}_{i,2} \oplus  \cdots \oplus m_{i,k_i}\mathcal{O}_{i,k_i})
\end{align*}
as a disjoint union of finitely many constructible subsets of Krull-Schmidt type, where $ \mathcal{O}_{i,j}\;( 1 \leqslant i \leqslant l,1 \leqslant j \leqslant k_i ) $ are indecomposable constructible subsets disjoint to each other. 
Hence, by the scissor relation in Definition \ref{Scissor relation} (1), we may always assume that $ \mathcal{O}_1, \cdots, \mathcal{O}_n $ are disjoint to each other when we consider the element in $ \mathcal{SDH}^{\mathrm{red}}_t(\mathcal{A}) $ which is supported on $ \mathcal{O} $.
\end{remark}

\begin{definition}
    \label{positive or negative}
    Let $ \mathcal{O} \subseteq C_2(Q, \underline{\alpha}) $ be a $ G_{\underline{\alpha}} $-invariant constructible subset of Krull-Schmidt type. 
    
     $ \mathcal{O} $ is called positive (resp. negative) if all points in it correspond to complexes of the form $ a_1 C_{M_1} \oplus \cdots \oplus a_n C_{M_n}  $ (resp. $  a_1 C_{M_1}^* \oplus \cdots \oplus a_n C_{M_n}^* $), where $ a_1, \cdots ,a_n \in \mathbb{N} $ and $ M_1,\cdots,M_n \in \mathcal{A} $ are indecomposable and non-isomorphic to each other.

\end{definition}

Any negative constructible set is of the form $ \mathcal{O}^* $ for some positive constructible set $ \mathcal{O} $, where 
\[\mathcal{O}^* = \{(y^0,y^1,-d^0,-d^1) \mid (y^1,y^0,d^1,d^0) \in \mathcal{O}\}.\]

Set $ \underline{\alpha}^* = (\alpha^0,\alpha^1) $ for any dimension vector pair $ \underline{\alpha} = (\alpha^1,\alpha^0) \in \mathbb{N}\mathbb{I} \times \mathbb{N}\mathbb{I}$.

Then, we give the following definition of the motivic Hall algebra of $ \mathcal{A} $ in the framework of \cite{Joyce07} as in Section \ref{Motivic Hall algebra of C2A}. 

For any $ \alpha \in \mathbb{N}\mathbb{I} $, let
\[\mathcal{R}_{\alpha} = [E_{\alpha}^{\mathrm{nil}}(Q)/G_{\alpha}(Q)] \]
be the quotient stack parameterizing the isomorphism classes of nilpotent $ \mathbb{C}Q $-modules of dimension vector $ \alpha $. 
Set
\[\mathcal{R} = \bigsqcup_{\alpha \in \mathbb{N}\mathbb{I}} \mathcal{R}_{\alpha},\]
then $ \mathcal{R} $ is an Artin stack, locally of finite type over $ \mathbb{C} $ with affine stabilizers. 

Note that there is a natural morphism of stacks
\[
    I^+:\mathcal R\to \mathcal M \text{  (resp. } I^-:\mathcal R\to \mathcal M)
\]
corresponding to the map $[M]\mapsto [C_M]\text{  (resp. } [M]\mapsto [C_M^*]$), 
then $ \mathcal{R} $ is identified with a locally closed substack of $ \mathcal{M} $ via $ I^{\pm} $. 

Let $ \mathcal{R}^{(2)} $ be the moduli stack parameterizing filtrations in $ \mathcal{A} $ and set
\[\mathcal{H}_t(\mathcal{A}) =  K_{\Upsilon}^0(\mathrm{St}/\mathcal{R}), \]
with the multiplication given by $ \Tilde{b}_* \circ (\Tilde{a}_1, \Tilde{a}_2)^* $, where $ (\Tilde{a}_1, \Tilde{a}_2) $ and $ \Tilde{b} $ are defined analogously as in Section \ref{Moduli stacks of objects and filtrations in C2A}, then $ \mathcal{H}_t(\mathcal{A}) $ is a $ \mathbb{N}\mathbb{I} $-graded $ \mathbb{C}(t) $-algebra by construction.

Let $ \mathcal{H}_t^{\mathrm{tw}}(\mathcal{A}) $ be the $ \mathbb{C}(t) $-algebra with underlying vector space the same as $ \mathcal{H}_t(\mathcal{A}) $ and multiplication twisted by $ (-t)^{\langle \alpha,\beta \rangle} $ for elements of degrees $ \alpha,\beta \in \mathbb{N}\mathbb{I} $. 

Clearly, there is an injective homomorphism of $ \mathbb{C}(t) $-algebras induced by $ I^+ $ (resp. $ I^- $) denoted by
\[ I^+_*: \mathcal{H}_t^{\mathrm{tw}}(\mathcal{A}) \hookrightarrow \mathcal{SDH}^{\mathrm{red}}_t(\mathcal{A}) \text{  (resp. } I^-_*:\mathcal{H}_t^{\mathrm{tw}}(\mathcal{A}) \hookrightarrow \mathcal{SDH}^{\mathrm{red}}_t(\mathcal{A}) ).\]

Finally, we prove a triangular decomposition of the motivic semi-derived Hall algebra following the proof of \cite[Theorem 3.25]{LP21}, see also \cite[Lemma 4.7]{Bridgeland13}, using Remark \ref{q = t^2} $ (2) $ and Remark \ref{disjoint union}. 

Note that using the short exact sequences given in the proof of \cite[Proposition 3.12]{LP21}, for any $\begin{tikzcd} X = (X^1 \arrow[r,"d^1"] & X^0) \arrow[l, shift left,"d^0"] \end{tikzcd} \in \mathcal{C}_2(\mathcal{A}) $, we have
    \begin{align*}
        \delta_{[X]} = f_t([X]) \delta_{[C_{H^0(X)} \oplus C_{H^1(X)}^*]} * b_{\widehat{\mathrm{Im}\;d^1}-\widehat{\mathrm{Im}\;d^0}} \in\mathcal{SDH}^{\mathrm{red}}_t(\mathcal{A}) , 
    \end{align*}
where $ f_t([X]) = (-t)^{-\langle \widehat{H^0(X)}- \widehat{H^1(X)},\widehat{\mathrm{Im}\;d^0} - \widehat{\mathrm{Im}\;d^1}\rangle} \Upsilon(\mathrm{Aut}_{\mathcal{C}_2(\mathcal{A})}(C_{H^0(X)} \oplus C_{H^1(X)}^*))\Upsilon(\mathrm{Aut}_{\mathcal{C}_2(\mathcal{A})}(X))^{-1} $.  

The following lemma is proved analogously to the claim in \cite[Theorem 3.25]{LP21} by replacing the cardinality of sets with the dimension of vector spaces. 

\begin{lemma}[{\itshape\cite[Theorem 3.25]{LP21}}]\;
\label{Inequality for End}

Let $ M, N \in \mathcal{A} $, consider the isomorphism $ \mathrm{Ext}^1_{\mathcal{C}_2(\mathcal{A})}(C_M,C_N^*) \cong \mathrm{Hom}_{\mathcal{A}}(M, N) $. 
For any $ g \in \mathrm{Hom}_{\mathcal{A}}(M,N) $, there is a short exact sequence \begin{tikzcd}
    0 \to C_N^* \to (N \arrow[r,"0"] & M \arrow[l, shift left,"g"]) \to C_M \to 0 
\end{tikzcd} 
with an induced triangle $  C_N^* \to C_{\mathrm{Ker}g} \oplus C_{\mathrm{Coker}g}^*  \to C_M \to C_N $ in $ \mathcal{D}_2(\mathcal{A}) $, where $ \mathcal{D}_2(\mathcal{A}) $ is the derived category of $ \mathbb{Z}/2 $-graded complexes. 
Then
    \begin{align*}
        \mathrm{dim}_{\mathbb{C}}\;\mathrm{End}_{\mathcal{D}_2(\mathcal{A})}(C_{\mathrm{Ker}g} \oplus C_{\mathrm{Coker}g}^*) \leqslant \mathrm{dim}_{\mathbb{C}}\;\mathrm{End}_{\mathcal{D}_2(\mathcal{A})}(C_M \oplus C_N^*)
    \end{align*}
    for any $ g \in \mathrm{Hom}_{\mathcal{A}}(M, N) $, and the equality holds if and only if $ g = 0 $.  
\end{lemma}

\begin{definition}
    \label{Grading of End in C2A}
    For any complex $ X \in \mathcal{C}_2(\mathcal{A}) $, define
     \begin{align*}
        \varepsilon([X]) = \varepsilon([C_{H^0(X)} \oplus C_{H^1(X)}^*]) = \mathrm{dim}_{\mathbb{C}}\;\mathrm{End}_{\mathcal{D}_2(\mathcal{A})}(C_{H^0(X)} \oplus C_{H^1(X)}^*).
    \end{align*}
    In particular, $ \varepsilon([K]) = \varepsilon([0]) = 0 $ for $ K \in \mathcal{C}_{2,\mathrm{ac}}(\mathcal{A}) $.
    
    Fix a dimension vector pair $ \underline{\alpha} \in \mathbb{N}\mathbb{I} \times \mathbb{N}\mathbb{I} $. 
    For any $ [\mathcal{F} \xrightarrow{f} \mathcal{M}_{\underline{\alpha}} \hookrightarrow \mathcal{M}] \in \mathcal{SDH}^{\mathrm{red}}_t(\mathcal{A})$, define
\begin{align*}
    \varepsilon([\mathcal{F} \xrightarrow{f} \mathcal{M}_{\underline{\alpha}} \hookrightarrow \mathcal{M}]) = \mathrm{max} \{\varepsilon([X]) \mid  X \in \mathcal{C}_2(\mathcal{A}) \text{ s.t. }   \delta_{[X]} = [\mathrm{Spec}\;{\mathbb{C}} \xrightarrow{i_{[X]}} \mathcal{M}] \in \mathrm{Im}\;f(\mathbb{C}) \}.
\end{align*}

\end{definition}

For simplicity, if $ f $ is injective, we set $ \varepsilon(\mathcal{F}) = \varepsilon([\mathcal{F} \xhookrightarrow{f} \mathcal{M}_{\underline{\alpha}} \hookrightarrow \mathcal{M}])  $ and we also have $ \varepsilon([X]) = \varepsilon([\mathrm{Spec}\;\mathbb{C}\xrightarrow{i_{[X]}} \mathcal{M}]) $.

The following proposition can be viewed as a motivic and reduced version of \cite[Theorem 3.25]{LP21}. 

\begin{proposition}
    \label{Triangular decomposition of SDH_t}
    The multiplication map $ x \otimes e^{\alpha} \otimes  y \mapsto I^+_{*}(x) * b_{\alpha} * I^-_{*}(y) $ induces an isomorphism of $ \mathbb{C}(t) $-vector spaces
    \begin{align*}
         \mathcal{H}_t^{\mathrm{tw}}(\mathcal{A}) \otimes \mathbb{C}(t)[K(\mathcal{A})] \otimes \mathcal{H}_t^{\mathrm{tw}}(\mathcal{A})  \xrightarrow{\sim} \mathcal{SDH}^{\mathrm{red}}_t(\mathcal{A}).
    \end{align*}
\end{proposition}
\begin{proof}   
    Since we are considering representable morphisms and stacks of finite type, by the discussion above Lemma \ref{Grading of End in C2A}, we obtain that as a $ \mathbb{C}(t) $-vector space, $ \mathcal{SDH}^{\mathrm{red}}_t(\mathcal{A}) $ is spanned by elements of the form
    \[[\mathcal{F} \xrightarrow{f}[\mathcal{O}_1\oplus \mathcal{O}_2^*/G_{\underline{\alpha}_1+\underline{\alpha}_2^*}]\hookrightarrow \mathcal{M}] * b_{\alpha},\]
    where $ \mathcal{O}_i \subseteq C_2(Q, \underline{\alpha}_i) $ is a positive $ G_{\underline{\alpha}_i} $-invariant constructible subset, $ \underline{\alpha}_i = (\alpha_i^1,\alpha_i^0) \in \mathbb{N}\mathbb{I} \times \mathbb{N}\mathbb{I} $ for  $ i=1,2 $, and $ \alpha \in K(\mathcal{A}) $. 

    Let $ x_1= [\mathcal{F}_1 \xrightarrow{f_1}[\mathcal{O}_1/G_{\underline{\alpha}_1}]\hookrightarrow\mathcal{M}], x_2^*= [\mathcal{F}_2 \xrightarrow{f_2}[\mathcal{O}_2^*/G_{\underline{\alpha}_2^*}]\hookrightarrow\mathcal{M}] \in \mathcal{SDH}^{\mathrm{red}}_t(\mathcal{A}) $, where $ \mathcal{O}_i \subseteq C_2(Q, \underline{\alpha}_i) $ is a positive $ G_{\underline{\alpha}_i} $-invariant constructible subset and $ f_i $ factors through $ \mathcal{R}_{\alpha_i^0} $  for $ i=1,2 $. 
    Let
    \[[\mathcal{G} \xrightarrow{g}[\mathcal{L}/G_{\underline{\beta}}] \hookrightarrow \mathcal{M}] * b_{\alpha}\]
    be a term appearing in the product $ x_1 * x_2^* $ with nonzero coefficient $ c(\mathcal{L}) \in \mathbb{C}(t) $,  where $ \mathcal{L}\subseteq C_2(Q, \underline{\beta})  $ is a $ G_{\underline{\beta}} $-invariant constructible subset such that $ \mathcal{L} = \mathcal{L}_1 \oplus \mathcal{L}_2^* $ for some positive $ \mathcal{L}_1, \mathcal{L}_2 $, and $ \alpha \in K(\mathcal{A}) $. 
    By Lemma \ref{Inequality for End}, we have
    \begin{align*}
        \varepsilon([\mathcal{L}/G_{\underline{\beta}}]) \leqslant \varepsilon([\mathcal{O}_1 \oplus \mathcal{O}_2^*/G_{\underline{\alpha}_1+\underline{\alpha}_2^*}]) = \varepsilon_0.
    \end{align*}
    
    Moreover, there exists a finite decomposition of $ \mathcal{F}_1 \times \mathcal{F}_2 $, indexed by some finite set $ L $, as a refinement of the stratification $ M $ in Proposition \ref{Motivic Riedtmann-Peng}, such that
    \begin{align*}
        x_1 * x_2^* = \sum_{l \in L} c((\mathcal{O}_1 \oplus \mathcal{O}_2^*)_l) [\mathcal{F}_l\xrightarrow{f_l}[(\mathcal{O}_1 \oplus \mathcal{O}_2^*)_l/G_{\underline{\alpha}_1+\underline{\alpha}_2^*}]\hookrightarrow \mathcal{M}] + \text{ ( terms with $ \varepsilon < \varepsilon_0 $ )},
    \end{align*}
    where $ (\mathcal{O}_1 \oplus \mathcal{O}_2^*)_l \subseteq \mathcal{O}_1 \oplus \mathcal{O}_2^* $ is a constructible subset with nonzero $ c((\mathcal{O}_1 \oplus \mathcal{O}_2^*)_l) \in \mathbb{C}(t) $. 

    For any $ m \in \mathbb{N} $, let $ \mathcal{SDH}^{\mathrm{red}}_t(\mathcal{A})_m  $ be the $ \mathbb{C}(t) $-subspace of $ \mathcal{SDH}^{\mathrm{red}}_t(\mathcal{A}) $ spanned by elements of the form
    \begin{align*}
        [\mathcal{F}\xrightarrow{f}[\mathcal{O}_1\oplus \mathcal{O}_2^*/G_{\underline{\alpha}_1+\underline{\alpha}_2^*}] \hookrightarrow \mathcal{M}] * b_{\alpha} \text{ s.t. } \varepsilon([\mathcal{F}\xrightarrow{f}\mathcal{M}])= \varepsilon([\mathcal{O}_1\oplus \mathcal{O}_2^*/G_{\underline{\alpha}_1+\underline{\alpha}_2^*}])=m,
    \end{align*}
    then $ \mathcal{SDH}^{\mathrm{red}}_t(\mathcal{A}) = \bigoplus_{m \in \mathbb{N}}  \mathcal{SDH}^{\mathrm{red}}_t(\mathcal{A})_m  $ as a $ \mathbb{N} $-graded vector space.
    
    Thus, it suffices to prove the injectivity in the corresponding homogeneous component. 
    In other words, assume that
    \begin{align*}
        \sum_{i=1}^n a_i(t) [\mathcal{F}_i \xrightarrow{f_i}[\mathcal{O}_{i,1} \oplus \mathcal{O}_{i,2}^*/G_{\underline{\alpha}_{i,1}+\underline{\alpha}_{i,2}}] \hookrightarrow \mathcal{M}] * b_{\alpha_i} = 0 \in  \mathcal{SDH}^{\mathrm{red}}_t(\mathcal{A}),
    \end{align*}
    where $ a_i(t) \in \mathbb{C}(t) $, $ \mathcal{O}_{i,j} \subseteq C_2(Q, \underline{\alpha}_{i,j}) $ is a positive $ G_{\underline{\alpha}_{i,j}} $-invariant constructible subset, $ \underline{\alpha}_{i,j} \in \mathbb{N}\mathbb{I} \times \mathbb{N}\mathbb{I} $, and $ \alpha_i \in K(\mathcal{A}) $ for any $ 1 \leqslant i \leqslant n,\;j=1,2 $, we must have $ a_i(t) = 0 $ for all $ i $. 

    Using the scissor relation in Definition \ref{Scissor relation} (1), we may always assume that the constructible sets $ \mathcal{O}_{i,1} \oplus \mathcal{O}_{i,2}^* (1 \leqslant i \leqslant n)$ are disjoint to each other, i.e., the complexes corresponding to the supports of these $ n $ elements  $ [\mathcal{F}_i \xrightarrow{f_i} \mathcal{M}] $ are non-isomorphic. 
    Then $ a_i(t) = 0 $ for all $ i $ by a similar argument in \cite[Theorem 3.20]{LP21} according to Remark \ref{q = t^2} $ (2) $. 
    Thus, the multiplication map is injective. 

    For the surjectivity, by Remark \ref{disjoint union}, we may always assume that $ \mathcal{O}_1 \cap \mathcal{O}_2^* = \emptyset $. 
    Moreover, stratifying $ \mathcal{O}_1 \oplus \mathcal{O}_2^* $ if necessary, we may also assume that
    \begin{align*}
    &[[\mathcal{O}_1/G_{\underline{\alpha}_1}]\hookrightarrow \mathcal{M}] * [[\mathcal{O}_2^*/G_{\underline{\alpha}_2^*}] \hookrightarrow \mathcal{M}] \\
    =& c(\mathcal{O}_1 \oplus \mathcal{O}_2^*) [[\mathcal{O}_1/G_{\underline{\alpha}_1}] \times [\mathcal{O}_2^*/G_{\underline{\alpha}_2^*}]\xrightarrow{\iota} [\mathcal{O}_1 \oplus \mathcal{O}_2^*/G_{\underline{\alpha}_1+\underline{\alpha}_2^*}] \hookrightarrow \mathcal{M}] + \text{ ( terms with $ \varepsilon < \varepsilon_0 $ )},
    \end{align*}
    where  $ c(\mathcal{O}_1 \oplus \mathcal{O}_2^*) \in \mathbb{C}(t) $ is nonzero and $ \iota $ corresponds to the map $ ([X],[Y]) \mapsto [X \oplus Y] $. 
    Multiplying this equation with $ c(\mathcal{O}_1 \oplus \mathcal{O}_2^*)^{-1} $ and using induction on $ \varepsilon([\mathcal{O}_1 \oplus \mathcal{O}_2^*/G_{\underline{\alpha}_1+\underline{\alpha}_2^*}]) $, we can prove that as a $ \mathbb{C}(t) $-vector space,  $ \mathcal{SDH}^{\mathrm{red}}_t(\mathcal{A}) $ is spanned by elements of the form
\[[\mathcal{F}_1 \xrightarrow{f_1}[\mathcal{O}_1/G_{\underline{\alpha}_1}] \hookrightarrow \mathcal{M}]  * [\mathcal{F}_2 \xrightarrow{f_2}[\mathcal{O}_2^*/G_{\underline{\alpha}_2^*}] \hookrightarrow \mathcal{M}] *  b_{\alpha}, \]
    where $ \mathcal{O}_i \subseteq C_2(Q, \underline{\alpha}_i) $ is a positive $ G_{\underline{\alpha}_i} $-invariant constructible subset, $ \underline{\alpha}_i \in \mathbb{N}\mathbb{I} \times \mathbb{N}\mathbb{I} $ for  $ i=1,2 $, and $ \alpha \in K(\mathcal{A}) $, which finishes the proof. 

\end{proof}

\newpage

\section{Realization of enveloping algebras of Borcherds-Bozec algebras}
\label{Realization of enveloping algebras of Borcherds-Bozec algebras}

\subsection{Borcherds-Bozec algebras and Quantum Borcherds-Bozec algebras}\;
\label{Borcherds-Bozec algebras and Quantum Borcherds-Bozec algebras}

Let $ \mathbb{I} $ be a finite index set, $ A = (a_{ij})_{i,j \in \mathbb{I}} $ is called an (even) symmetric Borcherds-Cartan matrix if
\begin{enumerate}
    \item $ a_{ii} \in \{2,0,-2,-4,\cdots \} $;
    \item $ a_{ij} = a_{ji} \in \mathbb{Z}_{\leqslant 0} $ for any $ i \neq j $.
\end{enumerate}
Set
\begin{align*}
    \mathbb{I}^{\mathrm{re}} = \{ i \in \mathbb{I} \mid a_{ii}=2\}, \;\mathbb{I}^{\mathrm{im}} = \{ i \in \mathbb{I} \mid a_{ii} \leqslant 0\} \text{ and }
     \mathbb{I}^{\infty} = (\mathbb{I}^{\mathrm{re}} \times \{1\}) \cup (\mathbb{I}^{\mathrm{im}} \times \mathbb{Z}_{>0}).
\end{align*}
Write $ i $ instead of $ (i,1) $ for $ i \in \mathbb{I}^{\mathrm{re}} $. 

Let $ \mathcal{G} $ be  the (derived) Borcherds-Bozec algebra over $ \mathbb{C} $ associated with $ A $. 
The following definition is given by \cite[Proposition 2.1]{FKKB21}.

\begin{definition}
   The universal enveloping algebra $ U(\mathcal{G}) $ associated with an (even) symmetric Borcherds-Cartan matrix $ A $ is defined to be an associative algebra over $\mathbb{C} $ with $ 1 $ generated by the elements $ e_{il},f_{il} \;((i,l)\in \mathbb{I}^{\infty})  $ and $ h_i \;(i \in \mathbb{I}) $ with the following defining relations:
    \begin{align}
        & h_i h_j - h_j h_i = 0 \quad \text{ for } i,j \in \mathbb{I}, \label{Relation 1}\\
        & e_{ik} f_{jl} - f_{jl} e_{ik} = k \delta_{ij}\delta_{kl} h_i \quad \text{ for } (i,k),(j,l) \in \mathbb{I}^{\infty}, \label{Relation 2}\\
        & h_ie_{jl}-e_{jl}h_i = l a_{ij}e_{jl}, \; h_if_{jl}-f_{jl}h_i = -l a_{ij}f_{jl} \quad \text{ for } i \in \mathbb{I},(j,l) \in \mathbb{I}^{\infty}, \label{Relation 3}\\
        & \sum\limits_{k=0}^{1-la_{ij}}(-1)^k \begin{pmatrix}
            1-l{a_{ij}} \\
            k
        \end{pmatrix}
        e_{i}^{1-la_{ij}-k} e_{jl} e_i^k = 0 \quad \text{ for }  i \in \mathbb{I}^{\mathrm{re}}, i \neq (j,l) \in \mathbb{I}^{\infty}, \label{Relation 4}\\
        & \sum\limits_{k=0}^{1-la_{ij}}(-1)^k \begin{pmatrix}
            1-l{a_{ij}} \\
            k
        \end{pmatrix}
        f_{i}^{1-la_{ij}-k} f_{jl} f_i^k = 0 \quad \text{ for }  i \in \mathbb{I}^{\mathrm{re}}, i \neq (j,l) \in \mathbb{I}^{\infty}, \label{Relation 5}\\
        & e_{ik} e_{jl} - e_{jl} e_{ik} =  f_{ik} f_{jl} - f_{jl} f_{ik} = 0 \quad \text{ for } a_{ij} = 0.\label{Relation 6}
    \end{align}
\end{definition}

Let $ t $ be an indeterminate and for any $ n,m \in \mathbb{Z} $, denote
\begin{align*}
    [n]_{t} = \frac{t^{n}-t^{-n}}{t-t^{-1}}, [n]_t! = \prod_{k=1}^n [k]_t \text{ and } \begin{bmatrix}
        m \\ n
    \end{bmatrix}_{t} = \frac{[m]_t!}{[n]_t![m-n]_t!}.
\end{align*}
Set $  [n] = [n]_{-t} , [n]! = [n]_{-t}!   $ and $ \begin{bmatrix}
        m \\ n
    \end{bmatrix} =\begin{bmatrix}
        m \\ n
    \end{bmatrix}_{-t} $.

For any $ i \in \mathbb{I} $, denote $  $
\begin{align*}
    (-t)_{(i)} = (-t)^{\frac{a_{ii}}{2}}.
\end{align*}

The following definition is given by \cite[Definition 2.8]{Bozec15}, \cite[Definition 3.7]{Bozec16} and \cite[Definition 2.2, Lemma A.1]{FKKB21}, see also \cite[Definition 3.1]{L23}, and we substitute the variable by $ -t $ since we are considering its classical limit at $ t=-1 $.  

\begin{definition}
    The quantum Borcherds-Bozec algebra  $ U_{-t}(\mathcal{G}) $ associated with an (even) symmetric Borcherds-Cartan matrix $ A $ is defined to be an associative algebra over $ \mathbb{C}(t) $ with $ 1 $ generated by the elements $ e_{il},f_{il} \;((i,l)\in \mathbb{I}^{\infty}) $ and $  K_i^{\pm1} \;(i \in \mathbb{I}) $
    with the following defining relations:
    \begin{align}
        & K_i K_i^{-1} = K_i^{-1}K_i = 1,\; [K_i,K_j] = 0 \quad \text{ for } i,j \in \mathbb{I}, \label{Quantum Relation 1}\\
        & K_i e_{jl} = (-t)^{la_{ij}} e_{jl}K_i,\; K_i f_{jl} = (-t)^{-la_{ij}} f_{jl}K_i \quad \text{ for } i\in \mathbb{I},(j,l) \in \mathbb{I}^{\infty},\label{Quantum Relation 2}\\
        & \sum\limits_{k=0}^{1-la_{ij}} (-1)^k \begin{bmatrix}
            1-la_{ij}  \\ k
        \end{bmatrix} e_{i}^{1-la_{ij}-k} e_{jl} e_i^k = 0 \quad \text{ for }  i \in \mathbb{I}^{\mathrm{re}}, i \neq (j,l) \in \mathbb{I}^{\infty}, \label{Quantum Relation 3}\\
        & \sum\limits_{k=0}^{1-la_{ij}} (-1)^k\begin{bmatrix}
            1-la_{ij}  \\ k
        \end{bmatrix} f_{i}^{1-la_{ij}-k} f_{jl} f_i^k = 0 \quad \text{ for }  i \in \mathbb{I}^{\mathrm{re}}, i \neq (j,l) \in \mathbb{I}^{\infty},\label{Quantum Relation 4}\\
        & e_{ik}e_{jl} - e_{jl}e_{ik} = f_{ik}f_{jl} - f_{jl}f_{ik} = 0 \quad \text{ for } a_{ij} = 0, \label{Quantum Relation 5}\\
        & e_{ik}f_{jl} - f_{jl}e_{ik} = 0 \quad \text{ for } i \neq j,\label{Quantum Relation 6}\\
        & \sum\limits_{\substack{m+r=k\\
        r+s=l}}  (-t)_{(i)}^{r(m-s)} \nu_{ir} e_{is} f_{im} (K_i^{-1})^r = \sum\limits_{\substack{m+r=k\\
        r+s=l}} (-t)_{(i)}^{-r(m-s)} \nu_{ir} f_{im} e_{is} (K_i)^r \label{Quantum Relation 7}\\
        & \quad\quad\quad\quad\quad\quad\quad\quad\quad\quad\quad\quad\quad\quad\quad\quad\quad\quad\quad\quad \notag \text{ for } (i,k),(i,l) \in \mathbb{I}^{\infty}.
    \end{align}
\end{definition}

From now on, we assume that $ \nu_{ir} = \frac{1}{\varphi_r(t^2)} $ where $ \varphi_r(x) = (1-x)(1-x^2)\cdots(1-x^r)  $ for any $ (i,r) \in \mathbb{I}^{\infty} $.

Consider the quiver $ Q =(\mathbb{I},\mathbb{H},s,t) $ as given in the beginning of Section \ref{Motivic form of semi-derived Ringel-Hall algebras}, let $ g_i $ be the number of loops at vertex $ i $ and $ c_{ij} $ be the number of arrows in $ \mathbb{H} $ from $ i $ to $ j $ if $ i \neq j $. 

Let $ A_Q = (a_{ij})_{i,j \in \mathbb{I}} $ be the symmetric Borcherds-Cartan matrix associated with $ Q $, i.e., 
\begin{align*}
    a_{ij} = \begin{cases}
        2 - 2 g_i &\text{ if } i = j,\\
        -c_{ij}-c_{ji} &\text{ if } i \neq j,
    \end{cases}
\end{align*}
and $ \mathcal{G}_{Q} $ be the Borcherds-Bozec algebra associated with $ A_Q $.

Let $ S_i $ be the nilpotent simple module at vertex $ i \in \mathbb{I} $, then $ (\widehat{S_i}, \widehat{S_j}) = a_{i j} $ for any $ i,j \in \mathbb{I} $.

The following lemma can be viewed as a motivic version of \cite[Lemma 3.2]{L23}. 

\begin{lemma}
    \label{eikfil}
    
    For any $ k,l \geqslant 1 $, in $ \mathcal{SDH}_t^{\mathrm{red}}(\mathcal{A})$, we have
    \begin{align*}
        \delta_{[C_{S_i^{\oplus k}}]} * \delta_{[C_{S_i^{\oplus l}}^*]} &= \sum_{r=0}^{\mathrm{min}\{k,l\}} (-t)_{(i)}^{r(l-k)} \frac{t^{-2r(k+l-r)}}{\varphi_r(t^{-2})} \delta_{[C_{S_i^{\oplus (k-r)}} \oplus C_{S_i^{\oplus (l-r)}}^*]} * (b_{S_i}^{-1})^{*r},\\
        \delta_{[C_{S_i^{\oplus l}}^*]} * \delta_{[C_{S_i^{\oplus k}}]} &= \sum_{r=0}^{\mathrm{min}\{k,l\}} (-t)_{(i)}^{-r(l-k)} \frac{t^{-2r(k+l-r)}}{\varphi_r(t^{-2})} \delta_{[C_{S_i^{\oplus (k-r)}} \oplus C_{S_i^{\oplus (l-r)}}^*]} * (b_{S_i})^{*r}.
    \end{align*}
    
\end{lemma}
\begin{proof}
    We have $ \mathrm{Hom}_{\mathcal{C}_2(\mathcal{A})}(C_{S_i^{\oplus k}}, C_{S_i^{\oplus l}}^* ) = 0 $, $ \mathrm{Ext}^1_{\mathcal{C}_2(\mathcal{A})}(C_{S_i^{\oplus k}}, C_{S_i^{\oplus l}}^* ) \cong \mathrm{Hom}_{\mathcal{A}}(S_i^{\oplus k}, S_i^{\oplus l}) \cong \mathbb{C}^{kl}  $.
    Moreover, any $ Z_{\xi} \in \mathcal{C}_2(\mathcal{A}) $ such that $ \mathrm{Ext}^1_{\mathcal{C}_2(\mathcal{A})}(C_{S_i^{\oplus k}}, C_{S_i^{\oplus l}}^* )_{Z_{\xi}} \neq \emptyset $ is of the form $  Z_{\xi}^{(r)} \cong C_{S_i^{\oplus (k-r)}} \oplus C_{S_i^{\oplus (l-r)}}^* \oplus K_{S_i^{\oplus r}}^* $
    for some $ 0 \leqslant r \leqslant \mathrm{min}\{k,l\}$. 

    By \cite[Lemma 2.6]{Bridgeland12} and Lemma \ref{Upsilon(Aut)}, we have
    \begin{align*}
        \Upsilon(\mathrm{GL}_r(\mathbb{C})) &= (-t)^{r^2+\frac{r(r-1)}{2}}((-t)-(-t)^{-1})^r [r]!,\\
        \frac{\Upsilon(\mathrm{Aut}_{\mathcal{C}_2(\mathcal{A})}(C_{S_i^{\oplus (k-r)}} \oplus C_{S_i^{\oplus (l-r)}}^*)) }{\Upsilon(\mathrm{Aut}_{\mathcal{C}_2(\mathcal{A})}(C_{S_i^{\oplus k}})\Upsilon(\mathrm{Aut}_{\mathcal{C}_2(\mathcal{A})}(C_{S_i^{\oplus l}}^*))  } &= \frac{(-t)^{3r^2
        +r-3rk-3rl} }{((-t)-(-t)^{-1})^{2r}} (\begin{bmatrix}
            k \\ r
        \end{bmatrix})^{-1}
        (\begin{bmatrix}
            l \\ r
        \end{bmatrix})^{-1} ([r]!)^{-2}.
    \end{align*}    
    By Proposition \ref{KX=XK} and Corollary \ref{bAbB}, we have
    \begin{align*}
        &\Upsilon(\mathrm{Aut}_{\mathcal{C}_2(\mathcal{A})}(Z_{\xi}^{(r)})) \delta_{[Z_{\xi}^{(r)}]} \\
        =& (-t)_{(i)}^{r(l-k)}\Upsilon(\mathrm{Aut}_{\mathcal{C}_2(\mathcal{A})}(C_{S_i^{\oplus (k-r)}} \oplus C_{S_i^{\oplus (l-r)}}^*)) \delta_{[C_{S_i^{\oplus (k-r)}} \oplus C_{S_i^{\oplus (l-r)}}^*]} * (b_{S_i}^*)^{*r}
    \end{align*}
    using $ \langle C_{S_i^{\oplus (k-r)}} \oplus C_{S_i^{\oplus (l-r)}}^*,K_{S_i^{\oplus r}}^*\rangle = r(l-r)\langle \widehat{S_i},\widehat{S_i} \rangle  $ and $ \langle \widehat{S_i},\widehat{S_i} \rangle = 1-g_i $.

    Note that  $ \mathrm{Ext}^1_{\mathcal{C}_2(\mathcal{A})}(C_{S_i^{\oplus k}}, C_{S_i^{\oplus l}}^* )_{Z_{\xi}^{(r)}} \subseteq \mathbb{C}^{kl} $ can be viewed as a locally closed subset consisting of $ k \times l $ complex matrices of rank $ r $, then there exists a principle $ \mathrm{GL}_r(\mathbb{C}) $-bundle $ p: \mathrm{Ext}^1_{\mathcal{C}_2(\mathcal{A})}(C_{S_i^{\oplus k}}, C_{S_i^{\oplus l}}^* )_{Z_{\xi}^{(r)}} \to \mathrm{Gr}(r,k) \times  \mathrm{Gr}(r,l) $, where $ \mathrm{Gr}(r,k), \mathrm{Gr}(r,l) $ denote the $ r $-Grassmannians of $ \mathbb{C}^k, \mathbb{C}^l $, respectively, sending a matrix to its column space and row space. 
    Also, using the Schubert cell decomposition of a Grassmannian, we have
    \begin{align*}
        \Upsilon(\mathrm{Gr}(r,k)) = \prod_{j=1}^{r} \frac{t^{2(k-r+j)}-1}{t^{2j}-1} = (-t)^{r(k-r)}\begin{bmatrix}
            k \\ r
        \end{bmatrix}.
    \end{align*}

    Hence, 
    \begin{align*}
         \delta_{[C_{S_i^{\oplus k}}]} * \delta_{[C_{S_i^{\oplus l}}^*]}=&\sum_{r=0}^{\mathrm{min}\{k,l\}} (-t)_{(i)}^{r(l-k)} \Upsilon(\mathrm{GL}_r(\mathbb{C})) \Upsilon(\mathrm{Gr}(r,k)) \Upsilon(\mathrm{Gr}(r,l))\\
        &\frac{\Upsilon(\mathrm{Aut}_{\mathcal{C}_2(\mathcal{A})}(C_{S_i^{\oplus (k-r)}} \oplus C_{S_i^{\oplus (l-r)}}^*)) }{\Upsilon(\mathrm{Aut}_{\mathcal{C}_2(\mathcal{A})}(C_{S_i^{\oplus k}})\Upsilon(\mathrm{Aut}_{\mathcal{C}_2(\mathcal{A})}(C_{S_i^{\oplus l}}^*))  }\delta_{[C_{S_i^{\oplus (k-r)}} \oplus C_{S_i^{\oplus (l-r)}}^*]} * (b_{S_i}^*)^{*r} \\
        =& \sum_{r=0}^{\mathrm{min}\{k,l\}} (-t)_{(i)}^{r(l-k)} \frac{t^{-2r(k+l-r)}}{\varphi_r(t^{-2})} \delta_{[C_{S_i^{\oplus (k-r)}} \oplus C_{S_i^{\oplus (l-r)}}^*]} * (b_{S_i}^{-1})^{*r}.
    \end{align*}

    For the second equality, it suffices to observe that $ \langle C_{S_i^{\oplus (k-r)}} \oplus C_{S_i^{\oplus (l-r)}}^*,K_{S_i^{\oplus r}}\rangle = r(k-r)\langle \widehat{S_i},\widehat{S_i} \rangle  $ and thus $ -r(l+k-2r) + 2 r(k-r) = -r(l-k)  $. 
    The remaining computations are the same as that in the first equality. 
\end{proof}

The following lemma is entirely analogous to the finite field case, using Lemma \ref{eikfil} and induction on $ k $.

\begin{lemma}[{\itshape\cite[Lemma 3.3]{L23}}]\;
     \label{Commutation relation}
     
     For any $ k,l \geqslant 1 $, in $ \mathcal{SDH}_t^{\mathrm{red}}(\mathcal{A}) $, we have
     \begin{align*}
         &\sum_{r=0}^{\mathrm{min}\{k,l\}} (-t)_{(i)}^{r(l-k)} \frac{t^{-2r(k+l)+2r(r+1)}}{\varphi_r(t^2)} \delta_{[C_{S_i^{\oplus (k-r)}}]} * \delta_{[C_{S_i^{\oplus (l-r)}}^*]} * (b_{S_i}^{-1})^{*r} \\
         =&\sum_{r=0}^{\mathrm{min}\{k,l\}} (-t)_{(i)}^{{-r(l-k)}} \frac{t^{-2r(k+l)+2r(r+1)}}{\varphi_r(t^2)} \delta_{[C_{S_i^{\oplus (l-r)}}^*]} * \delta_{[C_{S_i^{\oplus (k-r)}}]} * (b_{S_i})^{*r}. 
     \end{align*}

\end{lemma}

Inspired by \cite[Theorem 3.4]{L23} and \cite[Theorem 3.2]{LLR23}, we get a motivic version:

\begin{theorem}
    \label{Realization of quantum BB algebra}
    There is an injective homomorphism of $ \mathbb{C}(t) $-algebras
    \[\mathcal{S}\Phi_t : U_{-t}(\mathcal{G}_Q) \hookrightarrow \mathcal{SDH}^{\mathrm{red}}_t(\mathcal{A}) \]
    given by
    \begin{align*}
        e_{il} \mapsto (-t)^{l^2-l} \delta_{[C_{S_i^{\oplus l}}]},\;f_{il} &\mapsto (-t)^{l^2-l} \delta_{[C^*_{S_i^{\oplus l}}]},\; K_i \mapsto b_{S_i} ,\; K_i^{-1} \mapsto b_{S_i}^* \text{ for } (i,l) \in \mathbb{I}^{\infty}.
    \end{align*}
\end{theorem}
\begin{proof}
    First, $ \mathcal{S}\Phi_t $ preserves the relations  (\ref{Quantum Relation 1}), (\ref{Quantum Relation 2}) and (\ref{Quantum Relation 7}) by Corollary \ref{balphaalpha'}, Corollary \ref{KACB} and Lemma \ref{Commutation relation}, respectively. 
    
    For relation (\ref{Quantum Relation 6}), if $ i \ne j $, then $ \mathrm{Hom}_{\mathcal{C}_2(\mathcal{A})}(C_{S_i^{\oplus k}}, C_{S_j^{\oplus l}}^*) = 0 = \mathrm{Hom}_{\mathcal{C}_2(\mathcal{A})}(C_{S_j^{\oplus l}}^*,C_{S_i^{\oplus k}}) $, thus by Lemma \ref{Upsilon(Aut)}, we have $ \delta_{[C_{S_i^{\oplus k}}]} * \delta_{[C_{S_j^{\oplus l}}^*]} = \delta_{[C_{S_i^{\oplus k}} \oplus C_{S_j^{\oplus l}}^*]}  = \delta_{[C_{S_j^{\oplus l}}^*]} * \delta_{[C_{S_i^{\oplus k}}]}  $. 

    Note that by Proposition \ref{Triangular decomposition of SDH_t}, $ \mathcal{SDH}_t^{\mathrm{red}}(\mathcal{A}) $ admits a triangular decomposition 
    \[\mathcal{SDH}^{\mathrm{red}}_t(\mathcal{A}) \cong \mathcal{H}_t^{\mathrm{tw}}(\mathcal{A}) \otimes \mathbb{C}(t)[b_{S_i}^{\pm 1}\mid i \in \mathbb{I}]  \otimes \mathcal{H}_t^{\mathrm{tw}}(\mathcal{A})\]
    since $ K( \mathrm{rep}^{\mathrm{nil}}_{\mathbb{C}}(Q)) \cong \mathbb{Z}\mathbb{I} $, while $ U_{-t}(\mathcal{G}_Q) $ admits a triangular decomposition by \cite[Theorem 3.2]{KangKim20}. 
    Then it remains to show that $ \mathcal{S}\Phi_t $ preserves relation (\ref{Quantum Relation 5}) and the quantum Serre relation (\ref{Quantum Relation 3}) (resp. (\ref{Quantum Relation 4})), and is injective on the positive (resp. negative) half. 
    The proof is the same as the proof of \cite[Proposition 3.9, Theorem 4.2]{Kang18} over finite fields, following Remark \ref{q = t^2} $ (2) $. 
   
\end{proof}

\subsection{Classical limit of \texorpdfstring{$ \mathcal{SDH}_t^{\mathrm{red}}(\mathcal{A}) $}{SDHtredA}}\;
\label{Classical limit of SDHt}

Inspired by the process of taking the classical limit of $ U_{-t}(\mathcal{G}_Q) $ at $ t=-1 $ in \cite{FKKB21} and the construction in \cite[Section 6.1]{FLX24}, we are going to construct a $ \mathbb{C} $-algebra from $ \mathcal{C}_2(\mathcal{A}) $ of $ \mathcal{A}= \mathrm{rep}^{\mathrm{nil}}_{\mathbb{C}}(Q) $ using the motivic semi-derived Ringel-Hall algebra $ \mathcal{SDH}_t^{\mathrm{red}}(\mathcal{A}) $.

Note that the Chevalley generators $ e_{il}, f_{il} \;((i,l) \in \mathbb{I}^{\infty}) $ have a complicated commutation relation (\ref{Quantum Relation 7}). 
Thanks to Bozec, there are primitive generators $ s_{il}, t_{il} \;((i,l) \in \mathbb{I}^{\infty}) $ with simpler relations. 
Set 
\[\tau_{il} = (s_{il},s_{il})_L = (t_{il},t_{il})_L ,\]
where $ (\cdot,\cdot)_L $ is the unique symmetric bilinear form given by \cite[Proposition 2.2]{Bozec15} and \cite[Proposition 3.2]{Bozec16}. 

Let \[\mathbb{C}_{-1} = \mathbb{C}[t]_{(t+1)} = \{\frac{f(t)}{g(t)} \in \mathbb{C}(t) \mid f(t),g(t) \in \mathbb{C}[t],g(-1) \neq 0 \}\] be the localization of $ \mathbb{C}[t] $ at the prime ideal $ (t+1) $, which is a local ring with a unique maximal ideal $ (t+1)\mathbb{C}_{-1}  $. 
Consider the evaluation morphism
\[\pi: \mathbb{C}_{-1} \to \mathbb{C}, \frac{f(t)}{g(t)} \mapsto \frac{f(-1)}{g(-1)}\]
which induces an isomorphism $ \mathbb{C}_{-1}/(t+1)\mathbb{C}_{-1} \cong \mathbb{C} $. 
Note that  \[ \pi (\Upsilon(\mathcal{F})) = \chi^{\mathrm{na}}(\mathcal{F}(\mathbb{C}))\]
for finite type algebraic stack $ \mathcal{F} $ with affine geometric stabilizers, where $ \chi^{\mathrm{na}}(\mathcal{F}(\mathbb{C})) $ is the naive Euler characteristic of $ \mathcal{F} $, see \cite[Section 5.1] {FLX24} or \cite[Example 2.12]{Joyce07} for details. 
For simplicity, denote it by $ \chi(\mathcal{F}) $ in the following.

\begin{definition}\;
    \label{Classical limit of SDH}
    \begin{enumerate}
        \item Define the $ \mathbb{C}_{-1} $ form $ \mathcal{SDH}_t^{\mathrm{red}}(\mathcal{A})_{\mathbb{C}_{-1}} $ to be the $ \mathbb{C}_{-1} $-subalgebra of $ \mathcal{SDH}_t^{\mathrm{red}}(\mathcal{A}) $ generated by elements of the forms
    \begin{align*}
    [\mathcal{F} \xrightarrow{f}[\mathcal{O}/G_{\underline{\beta}}]\hookrightarrow \mathcal{M}],\frac{1}{\tau_{il}(t^2-1)}[\mathcal{F} \xrightarrow{f}[\mathcal{O}^*/G_{\underline{\beta}^*}] \hookrightarrow \mathcal{M}],\;\;b_{\alpha},\;\frac{b_{\alpha}-1}{-t-1},
    \end{align*}
    where $ \mathcal{O} \subseteq C_2(Q, \underline{\beta}) $ is a positive $ G_{\underline{\beta}} $-invariant constructible subset, $\underline{\beta} \in \mathbb{N}\mathbb{I} \times \mathbb{N}\mathbb{I} $, $ (i,l) \in \mathbb{I}^{\infty} $, and $\alpha \in K(\mathcal{A}) $. 

    \item Define the classical limit of $ \mathcal{SDH}_t^{\mathrm{red}}(\mathcal{A})_{\mathbb{C}_{-1}} $ at $ t=-1 $ to be the $ \mathbb{C} $-algebra
    \[\mathcal{SDH}_{-1}^{\mathrm{red}}(\mathcal{A}) = \mathbb{C}_{-1}/(t+1)\mathbb{C}_{-1} \otimes_{\mathbb{C}_{-1}} \mathcal{SDH}_t^{\mathrm{red}}(\mathcal{A})_{\mathbb{C}_{-1}}.\]        
    \end{enumerate}
        
\end{definition}

There is a canonical $ \mathbb{C} $-linear map
    \begin{align*}
        {}^{-}:\mathcal{SDH}_t^{\mathrm{red}}(\mathcal{A})_{\mathbb{C}_{-1}} &\to  \mathcal{SDH}_{-1}^{\mathrm{red}}(\mathcal{A})\\
        \sum \frac{f(t)}{g(t)} a_1 * \cdots * a_s &\mapsto 1 \otimes \sum \frac{f(t)}{g(t)} a_1 * \cdots * a_s =\sum \frac{f(-1)}{g(-1)} \otimes a_1 * \cdots * a_s, 
    \end{align*}
    where $  \frac{f(t)}{g(t)} \in \mathbb{C}_{-1} $ and $ a_1, \cdots ,a_s $ are generators in $ \mathcal{SDH}_t^{\mathrm{red}}(\mathcal{A})_{\mathbb{C}_{-1}}  $. 
    
For any $ \alpha \in K(\mathcal{A}) $, set
\begin{align*}
    h_{\alpha} =(\frac{b_{\alpha}-1}{-t-1} )^- =1 \otimes  \frac{b_{\alpha}-1}{-t-1} \in \mathcal{SDH}_{-1}^{\mathrm{red}}(\mathcal{A}). 
\end{align*}

Using Corollary \ref{balphaalpha'}, after a similar computation as in \cite[Lemma 6.3]{FLX24}, we obtain the following relations in $ \mathcal{SDH}_{-1}^{\mathrm{red}}(\mathcal{A}) $.

\begin{lemma}
    \label{Properties of minus}
    For any $ \alpha,\alpha' \in K(\mathcal{A}) $ and $ x=[\mathcal{F} \xrightarrow{f} \mathcal{M}_{\underline{\beta}} \hookrightarrow \mathcal{M} ] \in \mathcal{SDH}_t^{\mathrm{red}}(\mathcal{A})_{\mathbb{C}_{-1}} $, we have
    \begin{align*}
    \overline{b_{\alpha}} =1,\; [h_{\alpha},h_{\alpha'}] = 0,\; h_{\alpha} + h_{\alpha'} = h_{\alpha+\alpha'},  [h_{\alpha},\overline{x}] = (\alpha,\beta^0-\beta^1)\overline{x},
    \end{align*}
    where $ \underline{\beta} = (\beta^1,\beta^0) \in \mathbb{N}\mathbb{I} \times \mathbb{N}\mathbb{I} $.
\end{lemma}

\subsection{Enveloping algebra arising from \texorpdfstring{$ \mathcal{C}_2(\mathcal{A}) $}{C2A} }\;
\label{Enveloping algebra arising from C2A}

In this section, we construct two $ \mathbb{C} $-algebras inside $ \mathcal{SDH}_{-1}^{\mathrm{red}}(\mathcal{A}) $. 
The former follows directly from the injective homomorphism at the generalized quantum group level; the latter will be proved to coincide with  $ \mathcal{SDH}_{-1}^{\mathrm{red}}(\mathcal{A}) $.

\subsubsection{}

Recall that we have an injective homomorphism $ \mathcal{S}\Phi_t $ in Theorem \ref{Realization of quantum BB algebra} and primitive generators $ s_{il},t_{il}\;((i,l) \in \mathbb{I}^{\infty}) $, set
\begin{align*}
    \mathfrak{s}_{il} = \mathcal{S}\Phi_t(s_{il}), \mathfrak{t}_{il} = \mathcal{S}\Phi_t(t_{il}) \text{ and } \mathfrak{T}_{il} = \frac{\mathcal{S}\Phi_t(t_{il})}{\tau_{il}(t^2-1)} \in \mathcal{SDH}^{\mathrm{red}}_t(\mathcal{A}).
\end{align*}

If $ i \in \mathbb{I}^{\mathrm{re}} $, then $ l = 1 $ and $ \mathfrak{s}_{i1} = \delta_{[C_{S_i}]} \in \mathcal{SDH}_t^{\mathrm{red}}(\mathcal{A})_{\mathbb{C}_{-1}} $, $ \mathfrak{T}_{i1} = \frac{1}{\tau_{i1}(t^2-1)} \delta_{[C_{S_i}^*]} \in \mathcal{SDH}_t^{\mathrm{red}}(\mathcal{A})_{\mathbb{C}_{-1}} $. 
Let $ i \in \mathbb{I}^{\mathrm{im}} $ and $ l \geqslant 1 $. 
For the classical limit at $ t=-1 $, we use a suitable renormalization of the primitive
generators adapted to the specialization, still denoted by $ s_{il} $ and $ t_{il} $, while $ \tau_{il} $ is understood with respect to this normalization. 
Thus, without changing the notation, we may assume that $  \mathfrak{s}_{il},\ \mathfrak{T}_{il}
    \in
    \mathcal{SDH}_t^{\mathrm{red}}(\mathcal A)_{\mathbb C_{-1}} $, and that the corresponding relations have the required specialization at $ t=-1 $. 

\begin{definition}\;
\label{Definition of SU-1}

\begin{enumerate}
        \item Define $ \mathcal{SU}_t(\mathcal{A})_{\mathbb{C}_{-1}} $ to be the $ \mathbb{C}_{-1} $-subalgebra of $ \mathcal{SDH}_t^{\mathrm{red}}(\mathcal{A}) $ generated by the elements
    \begin{align*}
        \mathfrak{s}_{il}, \mathfrak{T}_{il}\;((i,l) \in \mathbb{I}^{\infty}) \text{ and } b_{\alpha},\frac{b_{\alpha}-1}{-t-1}\;(\alpha \in K(\mathcal{A})).
    \end{align*}
    
       \item  Define
    \[\mathcal{SU}_{-1}(\mathcal{A}) =  \mathbb{C}_{-1}/(t+1)\mathbb{C}_{-1} \otimes_{\mathbb{C}_{-1}} \mathcal{SU}_t(\mathcal{A})_{\mathbb{C}_{-1}}.\]
    \end{enumerate}
    \end{definition}

\begin{theorem}
    \label{SPhi to SU-1}
    There is an isomorphism of $ \mathbb{C} $-algebras
    \[ \mathcal{S}\Phi: U(\mathcal{G}_Q) \xrightarrow{\sim} \mathcal{SU}_{-1}(\mathcal{A}) \]
    given by
    \begin{align*}
        e_{il} \mapsto \overline{\mathfrak{s}_{il}},\; f_{il} \mapsto \overline{\mathfrak{T}_{il}},\; h_i \mapsto h_{\widehat{S_i}} \text{ for } (i,l) \in \mathbb{I}^{\infty}.
    \end{align*}    
\end{theorem}

\begin{proof}
    First, $ \mathcal{S}\Phi $ preserves the relations (\ref{Relation 1}) and (\ref{Relation 3}) by Lemma \ref{Properties of minus}. 
    It follows from \cite[Proposition 2.8]{FKKB21} that $  s_{il}t_{jk} - t_{jk}s_{il} = \delta_{ij}\delta_{lk} \tau_{il}(K_i^l-K_i^{-l})  $ for any $ (i,l),(j,k) \in \mathbb{I}^{\infty} $ and the quantum relations (\ref{Quantum Relation 3})-(\ref{Quantum Relation 5}) hold also for the primitive generators $ s_{il},t_{il} $ in $  U_{-t}(\mathcal{G}_Q) $. 
    Then $ \mathcal{S}\Phi $ preserves the relations (\ref{Relation 4})-(\ref{Relation 6}) using the fact that $ \overline{\begin{bmatrix}
            1-la_{ij}  \\ k
        \end{bmatrix}_{-t}} = \begin{pmatrix}
            1-la_{ij}  \\ k
        \end{pmatrix} $. 
    Also, by Lemma \ref{Properties of minus},
    \begin{align*}
        \overline{\mathfrak{s}_{il}}\overline{\mathfrak{T}_{il}} - \overline{\mathfrak{T}_{il}}\overline{\mathfrak{s}_{il}}= \delta_{ij}\delta_{lk} \otimes (\frac{b_{l\widehat{S_i}}-1}{t^2-1}-\frac{b_{-l\widehat{S_i}}-1}{t^2-1}) =  \delta_{ij}\delta_{lk} h_{l\widehat{S_i}} = l\delta_{ij}\delta_{lk} h_{\widehat{S_i}},
    \end{align*}
    then $ \mathcal{S}\Phi $ is well-defined. 
    
    Let $ U_{-t}(\mathcal{G}_Q)_{\mathbb{C}_{-1}} $ be the $ \mathbb{C}_{-1} $-algebra of $ U_{-t}(\mathcal{G}_Q) $ generated by $ s_{il}, \frac{t_{il}}{\tau_{il}(t^2-1)} \;((i,l) \in \mathbb{I}^{\infty}) $ and $ K_i, K_i^{-1}\;(i\in \mathbb{I}) $ as in \cite[Definition 3.1]{FKKB21}, then the injective homomorphism $ \mathcal{S}\Phi_t $ in Theorem \ref{Realization of quantum BB algebra} induces an isomorphism of $ \mathbb{C}_{-1} $-modules $ U_{-t}(\mathcal{G}_Q)_{\mathbb{C}_{-1}} \xrightarrow{\sim} \mathcal{SU}_t(\mathcal{A})_{\mathbb{C}_{-1}} $. 
    Taking the classical limits at $ t=-1 $ on both sides and applying \cite[Theorem 4.3]{FKKB21}, we obtain that $ \mathcal{S}\Phi $ is an isomorphism.

\end{proof}

\subsubsection{}

Recall that in Section \ref{Moduli stacks of objects and filtrations in C2A} we have defined a constructible subset $ C_2^{\mathrm{ind}}(Q,\underline{\alpha}) \subseteq C_2(Q,\underline{\alpha})  $ for any $\underline{\alpha} \in \mathbb{N}\mathbb{I} \times \mathbb{N}\mathbb{I} $. 
Let
\[\mathcal{M}^{\mathrm{ind}}_{\underline{\alpha}} = [C_2^{\mathrm{ind}}(Q,\underline{\alpha})/G_{\underline{\alpha}}] 
\]
be the quotient stack parameterizing the isomorphism classes of $\mathbb{Z}/2 $-graded indecomposable complexes of dimension vector pair $ \underline{\alpha} $ and set
\[
\mathcal{M}^{\mathrm{ind}} = \bigsqcup_{\underline{\alpha} \in \mathbb{N}\mathbb{I} \times \mathbb{N}\mathbb{I} } \mathcal{M}^{\mathrm{ind}}_{\underline{\alpha}},
\]
then $ \mathcal{M}^{\mathrm{ind}} $ is a constructible substack of $ \mathcal{M} $. 

\begin{definition}\;
    \label{Definition of SDHind-1}
    \begin{enumerate}
        \item Define $ \mathcal{SDH}_t^{\mathrm{ind}}(\mathcal{A})_{\mathbb{C}_{-1}} $ to be the $ \mathbb{C}_{-1} $-subalgebra of $ \mathcal{SDH}_t^{\mathrm{red}}(\mathcal{A}) $ generated by elements of the forms
    \begin{align*}
    [[\mathcal{O}/G_{\underline{\beta}}]\hookrightarrow \mathcal{M}^{\mathrm{ind}}\hookrightarrow \mathcal{M}],\frac{1}{\tau_{il}(t^2-1)}[[\mathcal{O}^*/G_{\underline{\beta}^*}] \hookrightarrow \mathcal{M}^{\mathrm{ind}}\hookrightarrow \mathcal{M}],\;\;b_{\alpha},\;\frac{b_{\alpha}-1}{-t-1},
    \end{align*}
    where $ \mathcal{O} \subseteq C_2^{\mathrm{ind}}(Q, \underline{\beta}) $ is a positive indecomposable $ G_{\underline{\beta}} $-invariant constructible subset, $\underline{\beta} \in \mathbb{N}\mathbb{I} \times \mathbb{N}\mathbb{I} $, $ (i,l) \in \mathbb{I}^{\infty} $, and $\alpha \in K(\mathcal{A}) $. 

    \item Define
    \[\mathcal{SDH}_{-1}^{\mathrm{ind}}(\mathcal{A}) = \mathbb{C}_{-1}/(t+1)\mathbb{C}_{-1} \otimes_{\mathbb{C}_{-1}} \mathcal{SDH}_t^{\mathrm{ind}}(\mathcal{A})_{\mathbb{C}_{-1}}.\]      
    
    \end{enumerate}
        
\end{definition}

The following proposition shows that after passing to the classical limit, it is enough to work with the $ \mathbb{C} $-algebra $ \mathcal{SDH}_{-1}^{\mathrm{ind}}(\mathcal A)$. 
This gives a structural description of the classical limit $ \mathcal{SDH}_{-1}^{\mathrm{red}}(\mathcal A)$ at $ t=-1 $. 

\begin{proposition}
    \label{ind=red(SDH)}
    $ \mathcal{SDH}_{-1}^{\mathrm{ind}}(\mathcal{A})  \cong \mathcal{SDH}_{-1}^{\mathrm{red}}(\mathcal{A}) $ as $ \mathbb{C} $-algebras. 
\end{proposition}

In order to prove the above proposition, we introduce the following notation for simplicity and establish some lemmas in $ \mathcal{SDH}_{-1}^{\mathrm{red}}(\mathcal{A}) $. 

Let $ \mathcal{O} = m_1 \mathcal{O}_1 \oplus \cdots \oplus m_n \mathcal{O}_n $ be a positive $ G_{\underline{\beta}} $-invariant constructible subset, where $ \mathcal{O}_i \subseteq C_2^{\mathrm{ind}}(Q, \underline{\beta}_i) $ is a positive indecomposable $ G_{\underline{\beta}_i} $-invariant constructible subset for any $ 1 \leqslant i \leqslant n $ and $ \underline{\beta} = m_1 \underline{\beta}_1 + \cdots + m_n \underline{\beta}_n $ with $ m_i \in \mathbb{N} $. 
Define the element
\begin{align*}
    1_{[\mathcal{O}/G_{\underline{\beta}}]} = 1 \otimes [\prod_{i=1}^n[\mathcal{O}_i/G_{\underline{\beta}_i}]^{m_i} \xrightarrow{\iota}[\mathcal{O}/G_{\underline{\beta}}]  \hookrightarrow \mathcal{M}] \in \mathcal{SDH}_{-1}^{\mathrm{red}}(\mathcal{A}),
\end{align*}
where $ \iota $ corresponds to the map
\[([X_{1,1}],\cdots,[X_{1,m_1}],\cdots,[X_{n,1}],\cdots,[X_{n,m_n}]) \mapsto [X_{1,1} \oplus \cdots \oplus X_{1,m_1} \oplus \cdots  \oplus X_{n,1}\oplus \cdots \oplus X_{n,m_n}].\]

By Remark \ref{disjoint union}, we may always assume that $ \mathcal{O}_i \cap \mathcal{O}_j = \emptyset $ for $ i \neq j $ and we need to show that $ 1_{[\mathcal{O}/G_{\underline{\beta}}]} \in \mathcal{SDH}_{-1}^{\mathrm{ind}}(\mathcal{A}) $.

\begin{definition}
    \label{Grading of gamma in C2A+}
    For any complex $ X \in \mathcal{C}_2(\mathcal{A}) $ of the form $ C_M $ for some $ M \in \mathcal{A} $, let $ \gamma([X])  =\gamma([C_M]) $ be the number of indecomposable direct summands in the Krull-Schmidt decomposition of $ M $. 
    
    Fix a dimension vector pair $ \underline{\beta} \in \mathbb{N}\mathbb{I} \times \mathbb{N}\mathbb{I} $. 
    For any $ [\mathcal{F} \xrightarrow{f} [\mathcal{O}/G_{\underline{\beta}}] \hookrightarrow \mathcal{M}] \in \mathcal{SDH}_t^{\mathrm{red}}(\mathcal{A})_{\mathbb{C}_{-1}}$ with $ \mathcal{O} $ positive, define
\begin{align*}
    \gamma([\mathcal{F} \xrightarrow{f} [\mathcal{O}/G_{\underline{\beta}}] \hookrightarrow \mathcal{M}]) = \mathrm{max} \{\gamma([X]) \mid  X \in \mathcal{C}_2(\mathcal{A}) \text{ s.t. }   \delta_{[X]}  \in \mathrm{Im}\;f(\mathbb{C}) \}.
\end{align*}

\end{definition}

For simplicity, if $ f $ is injective, we set $ \gamma(\mathcal{F}) = \gamma ([\mathcal{F} \xhookrightarrow{f} [\mathcal{O}/G_{\underline{\beta}}]\hookrightarrow \mathcal{M}]) $.

\begin{lemma}
    \label{Inequality for gamma in C2A+}
    Let $ \mathcal{O}_i \subseteq  C_2(Q,\underline{\beta}_i) $ be a positive $ G_{\underline{\beta}_i} $-invariant constructible subset for $ i=1,2 $. 
    By Proposition \ref{Motivic Riedtmann-Peng}, in $ \mathcal{SDH}_{-1}^{\mathrm{red}}(\mathcal{A}) $, write
     \begin{align*}
    1_{[\mathcal{O}_1/G_{\underline{\beta}_1}]} * 1_{[\mathcal{O}_2/G_{\underline{\beta}_2}]} = 1 \otimes \sum_{m \in M} [[V_m \times E_m^1/G_m \ltimes E_m^0] \xrightarrow{\varphi_m} \mathcal{M}],
    \end{align*}
    where $ V_m \subseteq \mathcal{O}_1 \times \mathcal{O}_2 $ is a constructible subset such that if  $ v \in V_m $ corresponds to two complexes $ (X,Y) $, then $ \mathrm{Stab}_{G_m}(v) \cong \mathrm{Aut}_{\mathcal{C}_2(\mathcal{A})}(X) \times \mathrm{Aut}_{\mathcal{C}_2(\mathcal{A})}(Y) $, the dimensions of $ E_m^0 \cong \mathrm{Hom_{\mathcal{C}_2(\mathcal{A})}}(X,Y) $ and $ E_m^1 \cong \mathrm{Ext^1_{\mathcal{C}_2(\mathcal{A})}}(X,Y) $ are constant for various $ v \in V_m $, and $  \varphi_m = b\pi_m $ corresponds to the map $ (v,\xi) \mapsto [Z_{\xi}] $. 
   
    We have
    \begin{align*}
        \gamma([V_m \times E_m^1/G_m \ltimes E_m^0]) \leqslant \gamma([\mathcal{O}_1/G_{\underline{\beta}_1}]) + \gamma([\mathcal{O}_2/G_{\underline{\beta}_2}]) \text{ for any } m \in M.
    \end{align*}    
\end{lemma}

\begin{proof}
    Divide $ V_m \times E_m^1 $ into a disjoint union of $ G_m \ltimes E_m^0 $-invariant constructible subsets
    \begin{align*}
        V_m \times E_m^1 = \bigsqcup_{i=1}^3 (V_m \times E_m^1)_i,
    \end{align*}
    where
    \begin{align*}
        (V_m \times E_m^1)_1 &= \{(v,\xi)\mid  1 \leqslant \gamma([Z_{\xi}]) < \gamma([X])+\gamma([Y]) \text{ with } (X,Y) \text{ corresponding to } v \}, \\
        (V_m \times E_m^1)_2 &= \{(v,\xi)\mid  Z_{\xi} \cong X \oplus Y \text{ with } (X,Y) \text{ corresponding to } v \},
    \end{align*}
    and $ (V_m \times E_m^1)_3 $ is their complement. 
    For any $ i=1,2,3 $, define the following constructible subsets
    \begin{align*}
        (V_m)_i = \pi_1((V_m \times E_m^1)_i) \subseteq 
        V_m\text{ and } (E_m^1)_i = \pi_2((V_m \times E_m^1)_i) \subseteq E_m^1,
    \end{align*}
    where $ \pi_1 $ and $ \pi_2 $ denote the first and second projections, respectively. 
    For simplicity, set
    \[x_i = \sum_{m \in M} [[(V_m \times E_m^1)_i/G_m \ltimes E_m^0] \xrightarrow{(\varphi_m)_i} \mathcal{M}], \]
    where $ (\varphi_m)_i $ is the restriction map of $ \varphi_m $, then
    \begin{align*}
        1_{[\mathcal{O}_1/G_{\underline{\beta}_1}]} * 1_{[\mathcal{O}_2/G_{\underline{\beta}_2}]} = 1 \otimes \sum_{i=1}^3\sum_{m \in M} [[(V_m \times E_m^1)_i/G_m \ltimes E_m^0] \xrightarrow{(\varphi_m)_i} \mathcal{M}] = 1 \otimes \sum_{i=1}^3 x_i.
    \end{align*}
    
    It suffices to show that $ 1 \otimes x_3 = 0  $. 
    For any $ (v, \xi) \in  (V_m \times E_m^1)_3 $, we have $ \gamma([Z_{\xi}]) \geqslant \gamma([X])+\gamma([Y]) $ and $  Z_{\xi}  \not\cong X \oplus Y $. 
         
         By Lemma \ref{Upsilon(Aut)}, we have $ (t^2-1)^{\gamma([Z_{\xi}])} \mid \Upsilon(\mathrm{Aut}_{\mathcal{C}_2(\mathcal{A})}(Z_{\xi})) $, 
    \begin{align*}
        (t^2-1)^{\gamma([X])} \mid \Upsilon(\mathrm{Aut}_{\mathcal{C}_2(\mathcal{A})}(X)) \text{ and } (t^2-1)^{\gamma([Y])} \mid \Upsilon(\mathrm{Aut}_{\mathcal{C}_2(\mathcal{A})}(Y)).
    \end{align*}
    By Lemma \ref{Ext^1(X,Y)Z}, we have $ (t^2-1) \mid \Upsilon(\mathrm{Ext}^1_{\mathcal{C}_2(\mathcal{A})}(X,Y)_{Z_{\xi}}) $, then
    \begin{align*}
        (t^2-1) \mid \frac{\Upsilon(\mathrm{Ext}^1_{\mathcal{C}_2(\mathcal{A})}(X,Y)_{Z_{\xi}})\Upsilon(\mathrm{Aut}_{\mathcal{C}_2(\mathcal{A})}(Z_{\xi}))}{\Upsilon(\mathrm{Aut}_{\mathcal{C}_2(\mathcal{A})}(X))\Upsilon(\mathrm{Aut}_{\mathcal{C}_2(\mathcal{A})}(Y))}, 
    \end{align*}
    and thus
    \begin{align*}
         (t^2-1) \mid [[(E_m^1)_3/(\mathrm{Aut}_{\mathcal{C}_2(\mathcal{A})}(X) \times \mathrm{Aut}_{\mathcal{C}_2(\mathcal{A})}(Y)) \ltimes \mathrm{Hom}_{\mathcal{C}_2(\mathcal{A})}(X,Y)] \xrightarrow{(\varphi_m)_3} \mathcal{M}].
    \end{align*}
    Hence, $ 1 \otimes x_3 =0 $.
\end{proof}

\begin{lemma}
    \label{O^k in C2A+}
    Let $ \mathcal{O} \subseteq  C_2^{\mathrm{ind}}(Q,\underline{\beta}) $ be a positive indecomposable $ G_{\underline{\beta}} $-invariant constructible subset. 
    For any integer  $ k \geqslant 1 $,write
    \begin{align*}      {1_{[\mathcal{O}/G_{\underline{\beta}}]}}^{*k}  = 1 \otimes \sum_{m \in M}\sum_{n \in N_m}c_{mn}[[W_{mn} \times E_m^1/G_m \ltimes E_m^0] \xrightarrow{b \pi_m \Tilde{\tau}_{mn}}\mathcal{M}],
    \end{align*}
    where the notations are the same as in the proof of Proposition \ref{KX=XK}. 
    In $ \mathcal{SDH}_{-1}^{\mathrm{red}}(\mathcal{A}) $, we have
    \begin{align*}
         {1_{[\mathcal{O}/G_{\underline{\beta}}]}}^{*k}  = k! 1_{[k\mathcal{O}/G_{k\underline{\beta}}]} + (\text{ terms with $ 1 \leqslant \gamma <k $ } ).
    \end{align*}   
\end{lemma}
\begin{proof}
    Prove by induction on $ k $. 
    If $ k=1 $, clearly. 
    Suppose that the lemma holds for some $ k = k' \geqslant 1 $. 
    By the induction hypothesis,
    \begin{align*}
        {1_{[\mathcal{O}/G_{\underline{\beta}}]}}^{*(k'+1)}  &= k'! 1_{[k'\mathcal{O}/G_{k'\underline{\beta}}]} * 1_{[\mathcal{O}/G_{\underline{\beta}}]} + (\text{ terms with $ 1 \leqslant \gamma < k' $ } ) * 1_{[\mathcal{O}/G_{\underline{\beta}}]}. 
    \end{align*}
    By Lemma \ref{Inequality for gamma in C2A+},
    \begin{align*}
        {1_{[\mathcal{O}/G_{\underline{\beta}}]}}^{*(k'+1)} = k'! c 1_{[(k'+1)\mathcal{O}/G_{(k'+1)\underline{\beta}}]} + (\text{ terms with $  1 \leqslant \gamma <k'+1 $ } )
    \end{align*}
    for some  $ c \in \mathbb{C} $. 
    It suffices to show that $ c=k'+1 $, in other words,
    \begin{align*}
        1_{[k'\mathcal{O}/G_{k'\underline{\beta}}]} * 1_{[\mathcal{O}/G_{\underline{\beta}}]} = (k'+1) 1_{[(k'+1)\mathcal{O}/G_{(k'+1)\underline{\beta}}]} +(\text{ terms with $ 1 \leqslant  \gamma <k'+1 $ }  ). 
    \end{align*}
    Using the same notations as in Lemma \ref{Inequality for gamma in C2A+}, it suffices to show that
    \begin{align*}
        1 \otimes x_2 = 1 \otimes \sum_{m \in M} [[(V_m \times E_m^1)_2/G_m \ltimes E_m^0] \xrightarrow{(\varphi_m)_2} \mathcal{M}] = (k'+1) 1_{[(k'+1)\mathcal{O}/G_{(k'+1)\underline{\beta}}]},
    \end{align*}
    where $ (V_m \times E_m^1)_2 $ consists of $ (v,\xi) $ such that $ Z_{\xi} \cong X  \oplus Y $ with complexes $ X,Y $ corresponding to points in $ k'\mathcal{O},\mathcal{O} $. 

    To prove this, we use the following general fact.  
    Let $ W = W_1 \oplus \cdots \oplus W_a $, where $ W_1, \cdots, W_a \in \mathcal{C}_2(\mathcal{A}) $ are indecomposable complexes with $ w $ many distinct isomorphism classes of sizes $ a_1, \cdots, a_w $ such that $ a_1 + \cdots + a_w =  a $, and $ Y \in \mathcal{C}_2(\mathcal{A}) $ be an indecomposable complex. 
    By Lemma \ref{Upsilon(Aut)}, we have
    \begin{align*}
        & [[\mathrm{Ext}^1_{\mathcal{C}_2(\mathcal{A})}(W,Y)_{W \oplus Y}/(\mathrm{Aut}_{\mathcal{C}_2(\mathcal{A})}(W) \times \mathrm{Aut}_{\mathcal{C}_2(\mathcal{A})}(Y)) \ltimes \mathrm{Hom}_{\mathcal{C}_2(\mathcal{A})}(W,Y)] \xrightarrow{\varphi} \mathcal{M}] \\
        =& t^{-2 \mathrm{dim}\; \mathrm{Hom}_{\mathcal{C}_2(\mathcal{A})}(W,Y)} \frac{\Upsilon(\mathrm{Aut}_{\mathcal{C}_2(\mathcal{A})}(W \oplus Y))}{\Upsilon(\mathrm{Aut}_{\mathcal{C}_2(\mathcal{A})}(W))\Upsilon(\mathrm{Aut}_{\mathcal{C}_2(\mathcal{A})}(Y))}\delta_{[W \oplus Y]} \\
        =& t^{-2 \mathrm{dim}\; \mathrm{Hom}_{\mathcal{C}_2(\mathcal{A})}(W,Y)}\frac{\Upsilon(\mathrm{Aut}_{\mathcal{C}_2(\mathcal{A})}(W \oplus Y))\prod_{i=1}^a \Upsilon(\mathrm{Aut}_{\mathcal{C}_2(\mathcal{A})}(W_i))^{-1}\Upsilon(\mathrm{Aut}_{\mathcal{C}_2(\mathcal{A})}(Y))^{-1}}{\Upsilon(\mathrm{Aut}_{\mathcal{C}_2(\mathcal{A})}(W))\prod_{i=1}^a \Upsilon(\mathrm{Aut}_{\mathcal{C}_2(\mathcal{A})}(W_i))^{-1}}\delta_{[W \oplus Y]} \\
        =& t^{2l(W,Y)} \cdot \begin{cases}
            t^2(1+t^2+ t^4 + \cdots + t^{2s(i)})\delta_{[W \oplus Y]}   \text{ , if } Y \cong W_i \text{ for some }   1 \leqslant i \leqslant a,  \\
            \delta_{[W \oplus Y]} \text{ , otherwise, } 
            \end{cases}
    \end{align*}
    where $ \varphi $ corresponds to the map $ 0 \mapsto [W \oplus Y] $, $ s(i) $ is the multiplicity of $ W_i $ in $ W $ and
    \begin{align*}
        l(W,Y) = - \mathrm{dim}\; \mathrm{Hom}_{\mathcal{C}_2(\mathcal{A})}(W,Y)+l_{(W_1, \cdots, W_a,Y)}-l_{(W_1,\cdots,W_a)}.
    \end{align*}

     Now, let $ X_1, \cdots, X_n $ be arbitrary ordered non-isomorphic indecomposable complexes corresponding to $ n $ points in $ \mathcal{O} $, $ m_1, \cdots, m_n \in \mathbb{N} $ such that $ m_1 + \cdots + m
    _n = k'+1 $, then
    \[Z = m_1 X_1 \oplus m_2 X_2 \oplus\cdots \oplus m_n X_n \]
    corresponds to a point in $ (k'+1)\mathcal{O} $. 
    For any $ 1 \leqslant j \leqslant n $, set
    \begin{align*}
        Z \backslash X_j = \oplus_{i=1}^{j-1} m_i X_i \oplus (m_j-1)X_j \oplus_{i=j+1}^{n} m_i X_i,
    \end{align*}
    which corresponds to a point in $ k' \mathcal{O} $. 
    Taking $ W= Z\backslash X_j $ and $ Y = X_j $, we have
    \begin{align*}
        &[[(V_m \times E_m^1)_{2,Z}/G_m \ltimes E_m^0] \xrightarrow{(\varphi_m)_{2,Z}} \mathcal{M}] = \sum_{j=1}^n t^{2l(Z \backslash X_j,X_j)+2}(1+t^2+\cdots +t^{2(m_j-1)}) \delta_{[Z]},
    \end{align*}
    where $ (V_m \times E_m^1)_{2,Z} $ consists of $ (v,\xi) $ such that $ Z_{\xi} \cong Z $ and $ (\varphi_m)_{2,Z} $ is the restriction map of $ (\varphi_m)_2 $.
    Thus,
    \begin{align*}
        1 \otimes [[(V_m \times E_m^1)_{2,Z}/G_m \ltimes E_m^0] \xrightarrow{(\varphi_m)_{2,Z}} \mathcal{M}] = (\sum_{j=1}^n m_j)\delta_{[Z]}  = (k'+1)[\mathrm{Spec}\;\mathbb{C} \xrightarrow{i_{[Z]}} \mathcal{M}] .
    \end{align*}
    Since the choice of $ X_1, \cdots, X_n $ is arbitrary, we have
    \begin{align*}
        1 \otimes x_2 =& (k'+1)\otimes  [\bigsqcup_{m \in M}[(V_m)_{2}/G_m]\xrightarrow{((\iota_m)_{2})} [k'\mathcal{O}\oplus\mathcal{O}/G_{(k'+1)\underline{\beta}}]\hookrightarrow \mathcal{M}] \\
        =& (k'+1) \otimes [[\mathcal{O}/G_{\underline{\beta}}]^{k'} \times [\mathcal{O}/G_{\underline{\beta}}] \xrightarrow{\iota}[k'\mathcal{O}\oplus\mathcal{O}/G_{(k'+1)\underline{\beta}}]\hookrightarrow \mathcal{M}] \\
        =& (k' + 1)1_{[(k'+1)\mathcal{O}/G_{(k'+1)\underline{\beta}}]},
    \end{align*}
    where $ (V_m)_2 $ consists of $ v \in V_m $ such that $ (v,\xi) \in (V_m \times E_m^1)_2 $ and $ (\iota_m)_2 $ corresponds to the map $ ([X],[Y]) \mapsto [X \oplus Y] $.

\end{proof}

\begin{corollary}
    \label{Of Krull-Schmidt in C2A+}
    Let $ \mathcal{O}_i \subseteq C_2^{\mathrm{ind}}(Q,\underline{\beta}_i) $ be a positive indecomposable $ G_{\underline{\beta}_i}$-invariant constructible subset for any $ 1 \leqslant i \leqslant n  $. 
    
    Assume that $ \mathcal{O}_1,\cdots,\mathcal{O}_n $ are disjoint to each other. 
    Let $ m_1, \cdots, m_n \in \mathbb{N} $ with $ m_1 + \cdots m
    _n = m $ and set $ m_1 \underline{\beta}_1 + \cdots  + m_n \underline{\beta}_n  = \underline{\beta}$. 
    In $ \mathcal{SDH}_{-1}^{\mathrm{red}}(\mathcal{A}) $, we have 
    \begin{align*}
        (1_{[\mathcal{O}_1/G_{\underline{\beta}_1}]})^{*m_1} * \cdots * (1_{[\mathcal{O}_n/G_{\underline{\beta}_n}]})^{*{m_n}} =& (m_1)! \cdots (m_n)! 1_{[m_1 \mathcal{O}_1 \oplus \cdots \oplus m_n \mathcal{O}_n/G_{\underline{\beta}}]} \\
        &+ (\text{ terms with } 1 \leqslant \gamma < m) .
    \end{align*}
    Moreover, let $ l_1, \cdots, l_n \in \mathbb{N} $ with $ l_1 + \cdots l
    _n = l $ and set $ l_1 \underline{\beta}_1 + \cdots  + l_n \underline{\beta}_n  = \underline{\alpha}$, then 
    \begin{align*}
        1_{[m_1 \mathcal{O}_1 \oplus \cdots \oplus m_n \mathcal{O}_n/G_{\underline{\beta}}]} * 1_{[l_1 \mathcal{O}_1 \oplus \cdots \oplus l_n \mathcal{O}_n/G_{\underline{\alpha}}]} =& \prod_{i=1}^n\frac{(m_i+l_i)!}{(m_i)!(l_i)!} 1_{[(m_1+l_1) \mathcal{O}_1 \oplus \cdots \oplus (m_n+l_n) \mathcal{O}_n/G_{\underline{\beta}+\underline{\alpha}}]} \\
        &+  (\text{ terms with } 1 \leqslant \gamma < m+l). 
    \end{align*}
\end{corollary}

Combining Lemma \ref{Inequality for gamma in C2A+}, the following lemma is analogous to \cite[Proposition 11]{DXX10}.

\begin{lemma}
     \label{Of stratified Krull-Schmidt in C2A+}
     Let $ \mathcal{O}_i \subseteq  C_2(Q,\underline{\beta}_i) $ be a positive $ G_{\underline{\beta}_i} $-invariant constructible subset for $ i=1,2 $. 
     By Proposition \ref{Motivic Riedtmann-Peng}, in $ \mathcal{SDH}_{-1}^{\mathrm{red}}(\mathcal{A}) $, write
     \begin{align*}
        1_{[\mathcal{O}_1/G_{\underline{\beta}_1}]} * 1_{[\mathcal{O}_2/G_{\underline{\beta}_2}]} = 1 \otimes \sum_{m \in M} [[V_m \times E_m^1/G_m \ltimes E_m^0] \xrightarrow{\varphi_m} \mathcal{M}]. 
    \end{align*}
     Refining the stratification if necessary, we have
     \begin{align*}
         \gamma([V_m \times E_m^1/G_m \ltimes E_m^0]) \leqslant \gamma([\mathcal{O}_1/G_{\underline{\beta}_1}]) + \gamma([\mathcal{O}_2/G_{\underline{\beta}_2}]),
     \end{align*}
     and
     $ \mathrm{Im}\;\varphi_m $ is of Krull-Schmidt type for any $ m \in M $.
    
\end{lemma}

\begin{proof}
    By Lemma \ref{Inequality for gamma in C2A+}, write $  V_m \times E_m^1 = \bigsqcup_{i=1}^2 (V_m \times E_m^1)_i $ and we have $ 1_{[\mathcal{O}_1/G_{\underline{\beta}_1}]} * 1_{[\mathcal{O}_2/G_{\underline{\beta}_2}]} = 1 \otimes{\sum_{i=1}^2} x_i $.

    Fix a complex $ Z $ such that
    \begin{align*}
        (t^2-1) \nmid [[(V_m \times E_m^1)_{2,Z}/G_m \ltimes E_m^0] \xrightarrow{(\varphi_m)_{2,Z}}\mathcal{M}]
    \end{align*}
    for some $ m \in M $, where $ (V_m \times E_m^1)_{2,Z} $ consists of $ (v,\xi) $ such that $ Z_{\xi} \cong Z $ and $ (\varphi_m)_{2,Z} $  is the restriction map of $ (\varphi_m)_{2} $. 
    Then there exists some $ v \in V_m $ corresponding to two complexes $ X, Y $ such that
    \begin{align*}
        (t^2-1) \nmid [[\mathrm{Ext}^1_{\mathcal{C}_2(\mathcal{A})}(X,Y)_{Z}/(\mathrm{Aut}_{\mathcal{C}_2(\mathcal{A})}(X) \times \mathrm{Aut}_{\mathcal{C}_2(\mathcal{A})}(Y)) \ltimes \mathrm{Hom}_{\mathcal{C}_2(\mathcal{A})}(X,Y)] \xrightarrow{\varphi} \mathcal{M}],
    \end{align*} 
    where $ \varphi $ corresponds to the map $ \xi \mapsto [Z] $. 
    By a similar argument as in \cite[Lemma 2.2]{Riedtmann94} and \cite[Proposition 11]{DXX10}, we will show that all indecomposable direct summands of $ Z $ are extensions of direct summands of $ X $ and $ Y $. 
    
    Suppose $ Z = Z_1 \oplus \cdots \oplus Z_s $ for some indecomposable complexes $ Z_1, \cdots, Z_s $. 
    There is an $ \mathbb{C}^* $-action on $ Z $ given by
    \begin{align*}
        t \cdot (z_1, \cdots, z_s) = (tz_1, \cdots,t^sz_s),
    \end{align*}
    where $ t \in \mathbb{C}^* $, $ z_j \in Z_j$ for any $ 1 \leqslant j \leqslant s $. 
    A subspace $ Z_0 \subseteq Z $ is stable under this action if and only if $ Z_0 = (Z_1 \cap Z_0) \oplus \cdots \oplus (Z_s \cap Z_0) $. 
    Note that by Proposition \ref{Motivic Riedtmann-Peng}, the $ \mathbb{C} $-points of the quotient stack
    \[[\mathrm{Ext}^1_{\mathcal{C}_2(\mathcal{A})}(X,Y)_{Z}/(\mathrm{Aut}_{\mathcal{C}_2(\mathcal{A})}(X) \times \mathrm{Aut}_{\mathcal{C}_2(\mathcal{A})}(Y)) \ltimes \mathrm{Hom}_{\mathcal{C}_2(\mathcal{A})}(X,Y)]\]
    are isomorphism classes of filtrations $ (\mathrm{Im}\;a_{\xi} \subseteq Z_{\xi}) $ for any $ \xi \in \mathrm{Ext}^1_{\mathcal{C}_2(\mathcal{A})}(X,Y)_{Z} $. 
    There is an induced $ \mathbb{C}^* $-action on the set of such filtrations whose fixed points are of the forms $  (Z_0 \subseteq Z) $
    such that
    \begin{align*}
        Z_0 &= (Z_0 \cap Z_1) \oplus \cdots \oplus  (Z_0 \cap Z_s) \cong Y, \\
        Z/Z_0 &= Z_1/(Z_0 \cap Z_1) \oplus \cdots \oplus Z_s/(Z_0 \cap Z_s) \cong X.
    \end{align*}

    Divide $ \mathrm{Ext}^1_{\mathcal{C}_2(\mathcal{A})}(X,Y)_{Z} $ into the disjoint union of two $ (\mathrm{Aut}_{\mathcal{C}_2(\mathcal{A})}(X) \times \mathrm{Aut}_{\mathcal{C}_2(\mathcal{A})}(Y)) \ltimes \mathrm{Hom}_{\mathcal{C}_2(\mathcal{A})}(X,Y) $-invariant constructible subsets
    \begin{align*}
         \mathrm{Ext}^1_{\mathcal{C}_2(\mathcal{A})}(X,Y)_{Z} = E^1_{0,Z} \bigsqcup E^1_{1,Z},
    \end{align*}
    where $ E^1_{0,Z} $ consists of $ \xi $ such that $ (\mathrm{Im}\;a_{\xi} \subseteq Z_{\xi}) $ is fixed under the $ \mathbb{C}^* $-action and $ E^1_{1,Z} $ is its complement, denote by $ (\varphi_{Z})_i  $ the restriction map of $ \varphi_{Z} $. 
    Note that by construction $ \mathbb{C}^* $ acts freely on the second term, we have
    \begin{align*}
        (t^2-1) \mid [[E^1_{1,Z}/(\mathrm{Aut}_{\mathcal{C}_2(\mathcal{A})}(X) \times \mathrm{Aut}_{\mathcal{C}_2(\mathcal{A})}(Y)) \ltimes \mathrm{Hom}_{\mathcal{C}_2(\mathcal{A})}(X,Y)] \xrightarrow{(\varphi_{Z})_1 } \mathcal{M}],
    \end{align*}
    and thus
    \begin{align*}
        (t^2-1) \nmid [[E^1_{0,Z}/(\mathrm{Aut}_{\mathcal{C}_2(\mathcal{A})}(X) \times \mathrm{Aut}_{\mathcal{C}_2(\mathcal{A})}(Y)) \ltimes \mathrm{Hom}_{\mathcal{C}_2(\mathcal{A})}(X,Y)] \xrightarrow{(\varphi_{Z})_0 } \mathcal{M}],
    \end{align*}
    which implies that the fixed point set is non-empty. 
   
    Hence, for any $ 1 \leqslant j \leqslant s $, we have a short exact sequence
    \begin{align*}
         0 \to Z_0 \cap Z_j \to Z_j \to Z_j/(Z_0 \cap Z_j) \to 0,
    \end{align*}
    where $ Z_0 \cap Z_j $ is a direct summand of $ Y $ and $ Z_j/(Z_0 \cap Z_j) $ is a direct summand of $ X $. 

    By Remark \ref{disjoint union}, without loss of generality, we may assume that
    \begin{align*}
        \mathcal{O}_1 = m_1\mathcal{P}_1 \oplus \cdots \oplus m_k \mathcal{P}_k,\; \mathcal{O}_2 = n_1\mathcal{P}_1 \oplus \cdots \oplus n_k \mathcal{P}_k,
    \end{align*}
    where
    \begin{align*}
         k= \gamma([\mathcal{O}_1/G_{\underline{\beta}_1}]) + \gamma([\mathcal{O}_2/G_{\underline{\beta}_2}]), 
    \end{align*}
    $ \mathcal{P}_i \subseteq C^{\mathrm{ind}}_2(Q,\underline{\alpha}_i) $ is a positive indecomposable $ G_{\underline{\alpha}_i} $-invariant   constructible subset for any $ 1 \leqslant i \leqslant k $, $ \mathcal{P}_i \cap \mathcal{P}_j = \emptyset $ or $ \mathcal{P}_i = \mathcal{P}_j $ for any $ i \neq j $, and $ m_i, n_i \in \{0,1\} $ for any $ 1 \leqslant i \leqslant k $. 

    Note that by Corollary \ref{Of Krull-Schmidt in C2A+}, 
    \begin{align*}
        1 \otimes x_2 =1 \otimes \sum_{m \in M}[[(V_m \times E_m^1)_2/G_m \ltimes E_m^0] \xrightarrow{(\varphi_m)_2} \mathcal{M}] = z  1_{[(m_1+n_1) \mathcal{P}_1 \oplus \cdots \oplus (m_k+n_k) \mathcal{P}_k/G_{\underline{\beta}_1+\underline{\beta}_2}]}
    \end{align*}
    for some $ z \in \mathbb{C} $, so $  \mathrm{Im}\;(\varphi_m)_2 $ is of Krull-Schmidt type.

    For the remaining case, divide $ (V_m \times E_m^1)_1 $ into a disjoint union of $ G_m \ltimes E_m^0 $-invariant constructible subsets
    \begin{align*}
        (V_m \times E_m^1)_1 = \bigsqcup_{s=1}^{k-1} (V_m \times E_m^1)_{1,s},
    \end{align*}
    where $ (V_m \times E_m^1)_{1,s} $ consists of $ (v,\xi) $ such that $ \gamma([Z_{\xi}]) = s $, denote by $ (\varphi_m)_{1,s} $ the restriction map of $ (\varphi_m)_{1} $, then
    \begin{align*}
      1 \otimes x_1 =  1 \otimes \sum_{m \in M} \sum_{s=1}^{k-1}[[(V_m \times E_m^1)_{1,s }/G_m \ltimes E_m^0] \xrightarrow{(\varphi_m)_{1,s}}\mathcal{M}].
    \end{align*}
    For any $ 1 \leqslant s \leqslant k-1 $, according to the above analysis of $ \mathbb{C}^* $-action and using a similar argument as in the proof of \cite[Proposition 11]{DXX10}, $ \mathrm{Im}\;(\varphi_m)_{1,s} $ is of Krull-Schmidt type.

\end{proof}

\begin{proof}[\emph{\textbf{Proof of Proposition \ref{ind=red(SDH)}}}]
    It suffices to show that for any $\underline{\beta} \in \mathbb{N}\mathbb{I} \times \mathbb{N}\mathbb{I} $ and any positive $ G_{\underline{\beta}} $-invariant constructible subset $ \mathcal{O} \subseteq C_2(Q, \underline{\beta}) $ of the form  $ \mathcal{O} = m_1 \mathcal{O}_1 \oplus \cdots \oplus m_n \mathcal{O}_n $ with $ \mathcal{O}_i \subseteq C_2^{\mathrm{ind}}(Q,\underline{\beta}_i) $ positive indecomposable and $ \mathcal{O}_i \cap \mathcal{O}_j = \emptyset $ for any $ i \neq j $, we have $ 1_{[\mathcal{O}/G_{\underline{\beta}}]} \in \mathcal{SDH}^{\mathrm{ind}}_{-1}(\mathcal{A}) $. 

    Prove by induction on $ \gamma([\mathcal{O}/G_{\underline{\beta}}]) = m_1 + \cdots + m_n = m $. 
    If $ m =1 $, clearly. 
    Suppose that the statement holds for positive constructible subsets with $ \gamma \leqslant m-1 $. 
    Write $ \mathcal{O}' = (m_1-1)\mathcal{O}_1 \oplus \cdots \oplus m_n\mathcal{O}_n $ and set $ \underline{\beta}' = (m_1-1)\underline{\beta}_1 + \cdots + m_n \underline{\beta}_n $, then $ 1_{[\mathcal{O}'/G_{\underline{\beta}'}]} \in \mathcal{SDH}^{\mathrm{ind}}_{-1}(\mathcal{A}) $ by the induction hypothesis. 
    Consider $ 1_{[\mathcal{O}'/G_{\underline{\beta}'}]} * 1_{[\mathcal{O}_1/G_{\underline{\beta}_1}]} $ in $  \mathcal{SDH}^{\mathrm{ind}}_{-1}(\mathcal{A}) $. 
    
    By Lemma \ref{Inequality for gamma in C2A+} and the fact in the proof of Lemma \ref{O^k in C2A+}, we have
    \begin{align*}
         1_{[\mathcal{O}/G_{\underline{\beta}}]}= \frac{1}{m_1}{1_{[\mathcal{O}'/G_{\underline{\beta}'}]}} *  1_{[\mathcal{O}_1/G_{\underline{\beta}_1}]} - (\text{ terms with $ 1 \leqslant \gamma <m $ } ).
    \end{align*}
    Also, Lemma \ref{Of stratified Krull-Schmidt in C2A+} enables us to apply the induction hypothesis to the lower terms and thus $ 1_{[\mathcal{O}/G_{\underline{\beta}}]} \in \mathcal{SDH}^{\mathrm{ind}}_{-1}(\mathcal{A}) $.

\end{proof}

Combining Theorem \ref{SPhi to SU-1} and Proposition \ref{ind=red(SDH)}, we obtain the following. 
\begin{corollary}
    \label{SPhi to SDHind-1}
    There is an injective homomorphism of $ \mathbb{C} $-algebras
    \[ \mathcal{S}\Phi: U(\mathcal{G}_Q) \hookrightarrow \mathcal{SDH}^{\mathrm{ind}}_{-1}(\mathcal{A}) \]
    given by
    \begin{align*}
        e_{il} \mapsto \overline{\mathfrak{s}_{il}},\; f_{il} \mapsto \overline{\mathfrak{T}_{il}},\; h_i \mapsto h_{\widehat{S_i}} \text{ for } (i,l) \in \mathbb{I}^{\infty}.
    \end{align*}
\end{corollary}

\newpage

\bibliographystyle{amsplain}
\bibliography{REA}

\end{document}